\pdfoutput=1

\documentclass[11pt,appendix=true]{scrartcl}

\usepackage[utf8]{inputenc}

\usepackage{bbm} 
 \usepackage{amsmath}
 \usepackage{amssymb}
\usepackage{mathtools}
\usepackage[]{algorithm}
\usepackage{algpseudocode}
\usepackage{enumitem}

\usepackage{graphicx}
\usepackage{grffile}
\graphicspath{ {Figures/} }
\usepackage{subcaption}
\makeatletter
\algnewcommand\Input{\item[\textbf{Input:}]}%
\algnewcommand\Output{\item[\textbf{Output:}]}%
\algnewcommand\Init{\item[\textbf{Init:}]}%
\makeatother

\usepackage{xcolor}
\input{tudcolours.def}

\usepackage{amsthm}
\newtheorem{myTheorem}{Theorem}
\newtheorem{myProposition}{Proposition}

\newtheorem{myLemma}{Lemma}

\newtheorem{myRemark}{Remark}
\newtheorem{myDefinition}{Definition}
\newtheorem{myExample}{Example}

\newcommand{\setOfReals}{\mathbb{R}}
\newcommand{\setOfNaturals}{\mathbb{N}}
\newcommand{\setOfNonnegativeIntegers}{\mathbb{N}_0}
\newcommand{\setOfPositiveReals}{\setOfReals_{+}}
\newcommand{\setN}[1]{ [ #1]}

\newcommand{\setTime}[1]{ \mathcal{T}_{#1}}

\newcommand{\cardinality}[1]{\mid #1  \mid}

\newcommand{\coset}[2]{\langle #1 \rangle_{#2}}

\newcommand{\KL}[2]{D_\mathrm{KL}\left( #1 \mid\mid #2  \right) }

\newcommand{\indicator}[1]{\mathbbm{1}(#1)}

\newcommand{\diophantine}[2]{\Lambda (#1,#2)}

\newcommand{\SymmetryGroup}[1]{\mathsf{Sym}\left( #1 \right)}

\newcommand{\Aut}[1]{\mathsf{Aut}\left( #1 \right)}

\newcommand{\aggregation}[1]{\mathsf{agg}\left( #1  \right)}
\newcommand{\uniformization}[1]{\mathsf{unif}\left( #1  \right)}
\newcommand{\fibre}[1]{\mathsf{fibre}\left( #1 \right)}

\newcommand{\prob}{\mathsf{P}}
\newcommand{\probOf}[1]{\prob(#1)}

\newcommand{\eqstop}{.}
\newcommand{\eqcomma}{,}
\newcommand{\defeq}{\coloneqq}

\newcommand{\ie}{\textit{i.e.}}
\newcommand{\eg}{\textit{e.g.}}

\usepackage{todonotes}

\newboolean{draftversion}
\setboolean{draftversion}{true} 

\newboolean{longVersion}
\setboolean{longVersion}{true}  

\newcommand{\proofLocation}[2]{\ifthenelse{\boolean{longVersion}}{\text{#1}}{\text{#2}}}

\newcommand{\tdWasiur}[2][]{\ifthenelse{\boolean{draftversion}}{\todo[inline, color=blue!20, caption={2do}, #1]{\begin{minipage}{\textwidth-4pt}\emph{Remark Wasiur:}\\#2\end{minipage}}}{}}

\newcommand{\tdArnab}[2][]{\ifthenelse{\boolean{draftversion}}{\todo[inline, color=orange!20, caption={2do}, #1]{\begin{minipage}{\textwidth-4pt}\emph{Remark Arnab:}\\#2\end{minipage}}}{}}

\newcommand{\tdGeneral}[2][]{\ifthenelse{\boolean{draftversion}}{\todo[inline, color=green!20, caption={2do}, #1]{\begin{minipage}{\textwidth-4pt}\emph{ToDO:}\\#2\end{minipage}}}{}}

\usepackage{tikz}
\usetikzlibrary{positioning,fit,shapes,backgrounds,circuits.logic.US}

\tikzstyle{server}=[circle, line width=0.5pt, rounded corners=0.1mm, draw=black!100, fill=tud3a!100]
\tikzstyle{vertex}=[circle, line width=0.5pt, draw=black!100, fill=tud0a!50]
\tikzstyle{dispatcher} =[and gate US, line width=0.5pt, draw=black!100, fill=tud1a!100]
\tikzstyle{dotbox} = [draw=white, fill=white, rectangle,  inner sep=10pt, inner ysep=20pt]

\tikzset{three_sided/.style={
		draw=none,rectangle, 
		append after command={
			[shorten <= -0.5\pgflinewidth]
			([shift={(-1.5\pgflinewidth,-0.5\pgflinewidth)}]\tikzlastnode.north west)
			edge([shift={( 0.5\pgflinewidth,-0.5\pgflinewidth)}]\tikzlastnode.north east)
			([shift={( 0.5\pgflinewidth,-0.5\pgflinewidth)}]\tikzlastnode.north east)
			edge([shift={( 0.5\pgflinewidth,+0.5\pgflinewidth)}]\tikzlastnode.south east)
			([shift={( 0.5\pgflinewidth,+0.5\pgflinewidth)}]\tikzlastnode.south east)
			edge([shift={(-1.0\pgflinewidth,+0.5\pgflinewidth)}]\tikzlastnode.south west)
		}
	}
}

\usepackage{hyperref} 

\title{Approximate lumpability for Markovian agent-based models using local  symmetries}

\author{Wasiur~R.~KhudaBukhsh\footnote{Department of Electrical Engineering and Information Technology,
		Technische Universit\"{a}t Darmstadt, Germany,
		email:\href{mailto:wasiur.khudabukhsh@bcs.tu-darmstadt.de}{wasiur.khudabukhsh@bcs.tu-darmstadt.de}  }      \hspace{1.5mm}\href{https://orcid.org/0000-0003-1803-0470}{\includegraphics[width=3mm]{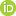}}, \emph{\small Technische Universit\"{a}t Darmstadt} \hfill 	 \\
	Arnab~Auddy\footnote{Indian Statistical Institute, 203 Barrackpore Trunk Road,
		Kolkata 700108,
		India, email:\href{mailto:arnabaudi@gmail.come}{arnabaudi@gmail.com} }, \emph{\small Indian Statistical Institute} \hfill \\
	Yann~Disser\footnote{Department of Mathematics,
		Technische Universit\"at Darmstadt,
		Dolivostrasse 15,
		64293 Darmstadt,
		Germany,
		email:\href{mailto:disser@mathematik.tu-darmstadt.de}{disser@mathematik.tu-darmstadt.de}  }     \hspace{1.5mm}\href{https://orcid.org/0000-0002-2085-0454}{\includegraphics[width=3mm]{Figures/ORCID-iD_icon-16x16}}, \emph{\small Technische Universit\"{a}t Darmstadt}  \hfill \\
	Heinz Koeppl\footnote{Department of Electrical Engineering and Information Technology,
		Technische Universit\"{a}t Darmstadt, Germany,
		email:\href{mailto:heinz.koeppl@bcs.tu-darmstadt.de}{heinz.koeppl@bcs.tu-darmstadt.de} }, \emph{\small Technische Universit\"{a}t Darmstadt}  \hfill
}

\date{}

\begin{document}
\maketitle


\begin{abstract}
We study 
a Markovian agent-based model (MABM) in this paper. Each agent is endowed with a local state that changes over time as the agent interacts with its neighbours. The neighbourhood structure is given by a graph. In a recent paper \cite{Simon2011Exact}, the authors used the automorphisms of the underlying graph to generate a lumpable partition of the joint state space ensuring Markovianness of the lumped process for binary dynamics. However, many large random graphs tend to become asymmetric rendering the automorphism-based lumping approach ineffective as a tool of model reduction. In order to mitigate this problem, we propose a lumping method based on a notion of local symmetry, which compares only local neighbourhoods of vertices. Since local symmetry only ensures approximate lumpability, we quantify the approximation error by means of Kullback-Leibler divergence rate between the original Markov chain and a \emph{lifted} Markov chain. We prove the approximation error decreases monotonically. The  connections to fibrations of graphs are also discussed.

\end{abstract}

\setcounter{equation}{0}
\section{Introduction}
\label{lump:sec:introduction}
 In this paper, a Markovian agent-based model (MABM)  refers to  a stochastic interacting particle system (IPS) with a \emph{finite} local state space. Given a  graph with $N$ vertices, we endow each vertex with a local state that varies  over time stochastically as the vertex interacts with its neighbours. Let $G=(V,E)$ be a graph (possibly a realisation of a random graph), where $V\defeq \{1,2, \ldots, N\}$ is the set of vertices, and $E \subseteq V \times V $ is the set of edges. For simplicity, we assume $G$ is undirected in the sense that $(u,v) \in E$ whenever $(v,u)\in E$, for $u,v \in V$.  
 Let $X_i(t)$ denote the local state of vertex~$i \in V$ at time~$t \in \setTime{} \defeq [0,T] $ for some $T >0$. For simplicity, we assume the vertices have the same finite  local state space $\mathcal{X} \defeq \{1,2,\ldots, K\}$ for some positive integer~$K$. 
 We assume the process~$X \defeq (X_1, X_2,\ldots, X_N ) \in \mathcal{X}^N$ is a continuous time Markov chain (CTMC), whose transition intensities depend on $G$. The objective of this paper is to devise an approximately lumpable partition of $\mathcal{X}^N$ \cite{kemeny1960finite,buchholz_1994,Deng2011KL,RUBINO1993WL_CTMC,rubino_1989}  using local symmetries of the graph $G$.


In a recent paper \cite{Simon2011Exact}, the authors introduced a novel lumping procedure based on the automorphisms of the underlying graph $G$. They considered a stochastic susceptible-infected-susceptible (SIS) epidemic process on a graph. They showed that, when the automorphism group is known, a lumpable partition can be obtained by  determining  the orbits of the elements of the state space with respect to the automorphism group.  The idea of lumping using graph automorphisms is innovative. However, it is   not always  efficient for two reasons. First,  finding all automorphisms without additional information about the graph structure is computationally prohibitive, especially for large graphs (see \cite{Babai2015Quasipolynomial}).  Second, there may be too few automorphisms to engender significant state space reduction \cite{Simon2011Exact} as many large random graphs tend to be asymmetric with high probability (see \cite{luczak1988automorphism,Kim2002RRG,McKay1984Auto}). Therefore, we propose a lumping procedure based on  a \emph{local}
notion of symmetry  \cite{Barbosa2016LocalSymmetry} taking  into account only local ($k$-hop) neighbourhoods of each vertex. In our approach, we construct  an equitable partition \cite[Chapter 9]{godsil2013algebraic} of $V$ by clubbing together vertices that are locally symmetric. We say two vertices $u$ and $v$ are locally symmetric if there exists an isomorphism $f$ between their respective local neighbourhoods (the induced subgraphs) such that $f(u)=v$. This is less restrictive than demanding the existence of an automorphism $g$ on the entire graph  $G$ mapping $u$ to $v$.

Local symmetry-driven lumping  allows for a more profitable aggregation than automorphism. Therefore, even when there are too few automorphisms, we can still achieve significant state space reduction by means of local symmetry-driven lumping. The price we pay for this gain is that the resultant lumped  process will only be \emph{approximately} Markovian. We
quantify the approximation error in terms of the Kullback-Leibler (KL) divergence  rate between 
the uniformization of the original process $X$ and a  Markov chain \emph{lifted}  from the lumped one (details provided in \autoref{lump:sec:localSymmetry}). 
 For a particular type of lifting called $\pi$-lifting, we prove that the approximation error decreases  as we increase the number of hops in our consideration  of local symmetries. 

Interestingly, the equivalent classes of the local symmetry can be shown to be  the same as the fibres of a graph fibration \cite{boldi2002fibrations}. Therefore, the fibres can also be used to aggregate the states of $\mathcal{X}^N$ to achieve approximate lumpability in the same fashion as we do with local symmetry. In addition to that, the problem of finding a lumpable partition for our MABM shares interesting connections with other related concepts in algebraic graph theory, such as colour refinement for directed graphs \cite{BerkholzBonsmaGrohe13,ArvindKoblerRattanVerbitsky16} and coverings \cite{Angluin80}. A discussion of these connections   paves the way for potential application of graph theoretic algorithms to problems in applied probability, and vice versa.

The paper is structured as follows. In \autoref{lump:sec:prelims}, we discuss some mathematical preliminaries required for the rest of the paper. We formally introduce the MABM in \autoref{lump:sec:MABM}. The lumping  based on graph automorphisms and related results are presented in \autoref{lump:sec:automorphismLumpCPRG}. In \autoref{lump:sec:localSymmetry}, we extend the lumping ideas to local symmetry of graphs. Connections to graph fibrations are explored in \autoref{lump:sec:fibrations}. The approximation error associated with local symmetry-driven lumping is studied in \autoref{lump:sec:approximation_error}.  Our theoretical discussions are also complemented with    some numerical results  on Erd\"os-R\'enyi, Barab\'asi-Albert preferential attachment, and Watts-Strogatz small world graphs.  
Finally, we conclude the paper with a short discussion in \autoref{lump:sec:discussion}.

\setcounter{equation}{0}
\section{Mathematical Preliminaries}
\label{lump:sec:prelims}

\paragraph{Notational conventions}
We use  $\setOfNaturals$ and $\setOfReals$ to denote the set of natural numbers and the set of real numbers. Also, we define $\setOfNonnegativeIntegers \defeq  \setOfNaturals \cup \{0\}$ and $\setOfPositiveReals \defeq  \setOfReals \setminus (-\infty,0]$. Additionally, we denote the set $\{1,2,\ldots, N\} $ by  $\setN{N}$. For a set $A$, we denote  its cardinality by $\cardinality{A}$, and    the class of all subsets of $A$, by $2^A$.  Given $N, K \in \setOfNaturals$, the set of all non-negative integer solutions to the Diophantine equation $x_1+x_2+\ldots+x_K=N$ by $\diophantine{N}{K}$, \ie, $\diophantine{N}{K} \defeq \{ x = (x_1,x_2,\ldots,x_K) \in \setOfNonnegativeIntegers^K \mid  x_1+x_2+\ldots+x_K=N \} $. We use $\indicator{.}$ to denote the indicator function. The symmetric group on a set $A$ is denoted by $\SymmetryGroup{A}$. 


\subsection{Lumpability}
We first define  lumpability for a discrete time Markov chain (DTMC) for ease of understanding. 
Standard references on this topic are \cite{kemeny1960finite,rubino_1989,RUBINO1993WL_CTMC,buchholz_1994}.

Let $\{Y(t)\}_{t \in \setOfNaturals}$ be a DTMC on a state space $\mathcal{Y}=  \setN{K}$ with transition probability matrix~$T = ((t_{i,j}))_{K\times K}$, where $t_{i,j} \defeq \probOf{ Y(2) =j \mid Y(1)=i }$.  Given a partition $\{ \mathcal{Y}_1, \mathcal{Y}_2,\ldots, \mathcal{Y}_M  \}$ of $\mathcal{Y}$, we define a process $\{Z(t) \}_{t \in \setOfNaturals}$ on $\setN{M}$ as follows: $Z(t) = i \in \setN{M}  \iff Y(t) \in \mathcal{Y}_i $, for each $t \in \setOfNaturals$. The process $Z$ is called the \emph{lumped} or the \emph{aggregated} process. The sets $\mathcal{Y}_i$'s are  often called lumping classes.

\begin{myDefinition}[Lumpability of a DTMC]
A DTMC $Y$   on a state space $\mathcal{Y}$ is  lumpable with respect to the partition $\{  \mathcal{Y}_1, \mathcal{Y}_2,\ldots, \mathcal{Y}_M \}$ of $\mathcal{Y}$, if the lumped process $Z$  is itself a DTMC for every choice of the initial distribution of $Y$ \cite[Chapter VI, p. 124]{kemeny1960finite}.
\end{myDefinition}

A necessary and sufficient condition for lumpability,  known as the Dynkin's criterion in the literature,  is the following: for any two pairs of lumping classes $\mathcal{Y}_i$ and $\mathcal{Y}_j$ with $i \neq j$, the transition probabilities  of moving into $\mathcal{Y}_j$  from any two states in $\mathcal{Y}_i$ are the same, \ie, $t_{u,  \mathcal{Y}_j } = t_{v,  \mathcal{Y}_j }$ for all $u, v \in \mathcal{Y}_i$, where we have used the shorthand notation $t_{u,  A} = \sum_{j  \in A } t_{u,j}$ for $A \subseteq \mathcal{Y}$. The common values, \ie, $\tilde{t}_{i,j} =   t_{u, \mathcal{Y}_{j} }$, for some $u \in \mathcal{Y}_i$, and $ i, j \in \setN{M}$,  form the transition probabilties of the lumped process $Z$.   Let  $\tilde{T} = ((  \tilde{t}_{i,j}))_{M\times M}$.  Since the Dynkin's criterion is both necessary and sufficient, some authors alternatively define lumpability in terms of Dynkin's criterion.   In the literature, the process $Z$ is sometimes denoted as $Z =\aggregation{Y}$.  

Now, we move to the continuous time case. The lumpability of a CTMC can be equivalently described in terms of lumpability of a linear system of ODEs. Consider the linear system $  \dot{y} =  y A  $,
where $A = ((a_{i,j}))$ is an $K\times K$ matrix (representing the transition rate 
matrix  of the corresponding continuous time Markov process). 

\begin{myDefinition}[Lumpability of a linear system]
  The linear system $  \dot{y} = y A  $ is said to be lumpable with respect to a partition $\{  \mathcal{Y}_1, \mathcal{Y}_2,\ldots, \mathcal{Y}_M \}$ of $\mathcal{Y}$, if there exists an $M\times K$ matrix $B = ((b_{i,j}))$ satisfying the Dynkin's criterion, \ie, if $b_{i,j} = \sum_{l \in \mathcal{Y}_j } a_{u,l}  =  \sum_{l  \in \mathcal{Y}_j } a_{v,l} $ for all $u, v \in \mathcal{Y}_i$.
\end{myDefinition}

%


An alternative approach to study lumpability of a CTMC is via its  uniformization. This approach will be particularly useful when we discuss  lumpability using local symmetries later. Let us consider a CTMC $\{Y(t)\}_{t \in \setTime{}}$ with transition rate matrix $A= ((a_{i,j}))$. It is known that the original CTMC is lumpable with respect to a given partition if and only if the uniformized DTMC is lumpable with respect to the same partition.   The uniformized DTMC $\tilde{Y}$ is often denoted by $\uniformization{Y}$, \ie, $  \tilde{Y} =\uniformization{Y}  $.    It was proved in \cite{RUBINO1993WL_CTMC,Ganguly2014CRN} that 
\begin{align}
  \label{lump:eq:unif_agg_equality}
\aggregation{\uniformization{Y}   } =\uniformization{\aggregation{Y}  }  \eqstop
\end{align}

Another useful observation that will be helpful later  is regarding  permutation of the states. It is intuitive that permutation  of elements of the state space does not destroy lumpability of a process. The proof of the following remark is straightforward, but is provided in \autoref{lump:sec:appendixA}   for the sake of completeness.

\begin{myRemark}
  \label{lump:proposition:permutation}
Let $Y$ be a CTMC on $\mathcal{Y}$ with transition rate matrix $A =((a_{i,j})) $. Let $f \in  \SymmetryGroup{ \mathcal{Y} }  $ be used to permute the states. If     $Y$ (or the linear system $ \dot{y}=y A  $) is lumpable with respect to a partition $\{ \mathcal{Y}_1, \mathcal{Y}_2,\ldots, \mathcal{Y}_M\}$, then the  process $Z= f(Y) $  is lumpable with respect to the  partition $\{ \mathcal{\tilde{Y}}_1, \mathcal{\tilde{Y}}_2,\ldots, \mathcal{\tilde{Y}}_M   \} $, where $  \mathcal{\tilde{Y}}_i = \{  f(u) \mid u \in  \mathcal{Y}_i  \}  $. 
\end{myRemark}

\setcounter{equation}{0}
\section{Markovian Agent-based Model}
\label{lump:sec:MABM}
%

\subsection{Interaction rules and the transition intensities}
\label{lump:subsec:intensities}
The most important ingredient of an MABM  are the interaction rules of the agent-based local processes $X_i$'s. These rules of interaction determine the dynamics of the process. Note that an MABM can also be viewed as a collection of local CTMCs that are connected to each other via the graph $G$.  In other words, each $X_i$ can be seen as a local CTMC, conditioned on the rest. In this work, we assume the intensities of the local CTMC $X_i$ depend on the  local states  $X_j$'s  of the neighbours of the vertex~$i \in V$ (such that $(i,j) \in E $). Let $d_i = \cardinality{\{ j \in V \mid  (i,j)\in E  \} }$ denote the number of neighbours of vertex~$i$. Additionally, we assume the intensities depend only on the counts of neighbours for each local state $a \in \mathcal{X}$. 
Therefore, we define the following summary function $c$ that returns population counts for different configurations of local states:
\begin{align*}
   c :  \{ \emptyset  \} \cup \left( \bigcup\limits_{l=1  }^{N}  \mathcal{X}^l  \right)  \longrightarrow   \{ \emptyset \} \cup \left(   \bigcup\limits_{l=1  }^{N}  \diophantine{l}{K}   \right)
\end{align*}
such that, for $x = (x_1,x_2,\ldots, x_l)  \in   \mathcal{X}^l$,  and  $l \in \setN{N}$,
\begin{align}
c( x) = (y_1,y_2,\ldots, y_K) \in \diophantine{l}{K} \text{ where } y_i =\quad \cardinality{ \{ x_j = i \in \mathcal{X}  : j =1,2, \ldots, l \}  }  \eqcomma
\end{align}
and set $c(\emptyset) = \emptyset$. 
The empty set $\emptyset$ is used to denote the neighbourhood of an isolated vertex. 
An important feature of the set-valued function $c$ is that it is permutation invariant in the sense  that $c(x)=c(x')$ if the elements of $x'$ are permutations of the elements of $x$. In order to extract the neighbourhood information out of the global configuration, we define a family of set-valued functions $n_i$ in the following way: 
\begin{align*}
  n_i : \mathcal{X}^N \longrightarrow  \{ \emptyset  \} \cup \left( \bigcup\limits_{l=1  }^{N-1}  \mathcal{X}^l  \right)  \, \text{ for } i \in \setN{N},
\end{align*}
such that, for $x=(x_1,x_2,\ldots, x_N) \in  \mathcal{X}^N$,
\begin{align}
n_i (x) =  \begin{cases}
(x_{i_1},x_{i_2}, \ldots, x_{i_l}) &  \text{ if } (i,i_j) \in E \, \forall j = 1,2,\ldots, l \, \text{ and } l=d_i, \\
    \emptyset & \text{ otherwise}\eqstop
\end{cases}
\end{align}
Having defined the these two important functions, we now define the interaction rules by means of local transition intensities. We assume the intensities depend only on the type of local transition and the summary of the neighbourhood configuration of a vertex. 
Therefore, we define the local intensity function 
\begin{align}
  \label{lump:eq:localIntensityFunction}
  \gamma :  \mathcal{X}\times \mathcal{X} \times  \left(  \{ \emptyset \} \cup \left(  \bigcup\limits_{l=1  }^{N-1}  \diophantine{l}{K}   \right) \right) \longrightarrow \setOfPositiveReals \eqcomma
\end{align}
where we interpret $ \gamma(a,b, y)  $ as the local intensity of making a transition from local state $a$ to $b$ by a vertex when the summary of its neighbourhood configuration is $y$. 

We are now in a position to specify the transition rate or the infinitesimal generator matrix for our MABM $X$. Note that the process $X$ jumps from a state $x$ to $y$ whenever one of the local processes $X_i$'s  jumps. Therefore, only one of the coordinates of the states $x$ and $y$ differ. Let the $K^N \times K^N$ matrix  $Q = ((q_{x,y}))$ denote the transition rate matrix of $X$. The elements of the matrix $Q$ are given by
\begin{align}
  q_{x,y}= \begin{cases}
  \sum_{i \in \setN{N}}  \indicator{ x_i \neq y_i, x_j=y_j \, \forall j \in V\setminus \{i\}  }   \gamma  (x_i, y_i, c(n_i (x))    )     & \text{ if } x \neq y \eqcomma \\
    - \sum_{y \neq x} q_{x,y} & \text{ if } x=y \eqstop
\end{cases}
\end{align}
We interpret $q_{x,y}$ as the rate of transition from $x$ to $y$, where $x,y \in \mathcal{X}^N$. For ease of understanding, we have suffixed the entries  of $Q$ by the different configurations $x,y \in \mathcal{X}^N$ and interpret them as functions on $\mathcal{X}^N\times \mathcal{X}^N$, instead of introducing a bijection between $\mathcal{X}^N$ and $\setN{ K^N  }$  to label the states in a linear order so that the suffixes range over the integers from $1$ to $K^N$. Note that  the particular choice of bijection to label the states is immaterial for our purposes, because such a  bijection essentially yields a permutation of $\setN{K^N}$, and  in the light of Proposition~\ref{lump:proposition:permutation}, does not alter lumpability properties of $Q$. Finally, we study  the dynamics of $X$ via the linear system
\begin{align}
  \label{lump:eq:KolmogorovEqn}
  \dot{p}  = p Q \eqstop
\end{align}
The vector-valued  function $p$  gives the probability distribution of $X$. 

%
%
\subsection{Examples}

\paragraph{Susceptible-infected-susceptible (SIS) epidemics}
The SIS epidemic model \cite{Simon2011Exact} captures the dynamics of an epidemic spread over a human or an animal population. It encapsulates binary dynamics in the sense that the local state space is written as $\mathcal{X} \defeq \{1,2\}$, where $1$ indicates susceptibility and $2$, an infected status.  Infected vertices infect one of its randomly chosen neighbours at each ticking of a Poisson clock with a fixed rate~$a>0$. Infected vertices themselves recover to susceptibility at a rate $b\geq 0$, independent of the neighbours' statuses. When $b=0$, the model is called a susceptible-infected (SI) model. 
Therefore, the local transition intensities are given by
\begin{align*}
  \gamma(1,2, (x_1,x_2)) = x_2 a, \quad \text{ and } \quad   \gamma(2,1, (x_1,x_2)) =b \eqstop
\end{align*}
We set $\gamma$ to zero in every other case. This fully describes the dynamics of the system. 

\paragraph{Peer-to-peer live media streaming systems}
Peer-to-peer networks are engineered networks where the vertices, called peers, communicate with each other to perform certain tasks in a distributed fashion. In particular,  content delivery platforms such as BitTorent, file sharing platforms such as Gnutella, (live) media (audio/video)  streaming platforms use peer-to-peer networks. For the purposes of performance analysis, Markov chain models are often used for such systems. 

In a peer-to-peer live streaming system, each peer maintains a buffer of length~$L$. The availability of a media chunk at buffer index $i\in \setN{L}$ is indicated by $1$, and likewise unavailability, by $0$ (see \cite{KhudaBukhsh2016P2PNetworking}). Therefore, local state space is given by $\mathcal{X}=\{0,1\}^L$. Put $K=2^L$ so that $\{0,1\}^L$ can be put in one-to-one correspondence with $\setN{K}$.  The chunk at buffer index~$L$, if available, is played back at rate unity and then removed. After playback, all other chunks are moved one index to the right, \ie, the chunk at buffer index~$i$ to shifted to buffer index~$i+1$. The central server selects a finite number of peers at random and uploads chunks at buffer index~$1$. All other peers (not receiving chunks from the server) download chunks from their neighbours, following a \emph{pull} mechanism\footnote{There are also systems where the peers \emph{push} chunks into their neighbours' buffers instead of pulling.}. The peers maintain their private Poisson clocks at the tickings of which they contact their neighbours to download missing chunks. Let the rate of these Poisson clocks be $a>0$. The neighbours oblige the request if the requested chunk is available. When multiple chunks are missing, the peers prioritise the chunks in some way giving rise to different chunk selection strategies, such as the latest deadline first (LDF) and the earliest deadline first (EDF) strategies. Let us introduce a function, called chunk selection function that captures this prioritisation, usually represented as probabilities. Let $s : \setN{L}\times \mathcal{X} \times \mathcal{X}$ be the chunk selection function. We interpret  $s(i, u,v)$ as the probability of a vertex with buffer configuration~$u$ selecting to fill buffer index~$i$ when it contacts a neighbour with buffer configuration~$v$.  Let $y_1,y_2,\ldots,y_K$ be a linear arrangement of the states in $\mathcal{X}$. Denote the $j$-th component of $y_i$ by $y_{i,j}$, \ie, $y_i=(y_{i,1}, y_{i,2}, \ldots, y_{i,L})$. The local intensity function is then given by \cite{KhudaBukhsh2016P2PNetworking}
\begin{align*}
  \gamma(u,u+e_j, (x_1,x_2,\ldots,x_K)) =
    a \sum_{ i \in \setN{K}}   \indicator{ y_{i,j}=1  }  x_i    s(j,u,y_i) & \text{ if } j >1 \eqcomma
\end{align*}
where $e_j$ is the $j$-th unit vector in the $L$-dimensional Euclidean space, and  $ (x_1,x_2,\ldots,x_K)  $ is the population count vector of the neighbours of a vertex with different buffer configurations. Besides the above transitions due to download of a chunk from a neighbour, there are two other transitions, namely, the transition due to the shifting after playback that takes place at rate unity irrespective of the buffer configurations of the neighbours, and the transition due to being directly served by the server. The latter event also takes place irrespective of the buffer configurations of the neighbours, but a rate that depends on the exact implementation setup of the peer-to-peer system. See  \cite{KhudaBukhsh2016P2PNetworking} for a detailed account  on this.




\setcounter{equation}{0}
\section{Automorphism-based lumping of an MABM}
\label{lump:sec:automorphismLumpCPRG}
Now we discuss how graph automorphisms can be used to lump states of $X$. The idea was introduced by \cite{Simon2011Exact}) for SIS epidemics on graphs. The purpose of lumping states is to generate a Markov chain on a smaller state space. However, we should make sure that the loss of information is not too much. For instance,  $X$ is always lumpable with respect to the partition $\{\mathcal{X}^N\} $, but   if all states are lumped together, all information about the dynamics of $X$ are lost except for the fact that total probability is conserved at all times. On the other hand, $X$ is also lumpable with respect to the partition $ \{   \{x\} \mid x \in \mathcal{X}^N  \} $, which retains all the information but does not yield any state space reduction. Therefore, one needs to find a meaningful partition that yields as much state space reduction as possible with minimal loss of information. For an MABM, population counts are very useful quantities. Therefore, in order to retain information about the population counts, we first partition $\mathcal{X}^N$ into  $\{    \mathcal{X}_a \mid a \in \diophantine{N}{K}  \}$, \ie,
\begin{align}
  \label{lump:eq:basic_partition}
  \mathcal{X}^N = \cup_{a \in \diophantine{N}{K}   }   \mathcal{X}_a  \text{ where }     \mathcal{X}_a  \defeq \{  b \in   \mathcal{X}^N \mid c(b)=a    \}   \eqcomma
\end{align}
 and then seek  a lumpable partition that is ideally  \emph{minimally}  finer than this. The partition in \autoref{lump:eq:basic_partition}   lumps together states that produce the same population counts.  The size of this partition, \ie,  $ \cardinality{   \{    \mathcal{X}_a \mid a \in \diophantine{N}{K}  \}   }  $, is $\binom{N+K-1}{K-1} $.  Note that, in the standard mean-field approach, one assumes that $X$ is lumpable with respect to the partition in \autoref{lump:eq:basic_partition} and studies (approximate) master equations (Kolmogorov forward equations) corresponding to the different population counts.   Next, we refine this partition using  automorphisms.


A bijection $f : V \longrightarrow V$ is called an automorphism on $G$ if $(i,j)\in E$ if and only if $(f(i), f(j)) \in E$, for all $i,j \in V$  (see \cite{godsil2013algebraic}). 
The collection of all automorphisms forms a group under the composition of maps. This group is denoted by $\Aut{G}$. Clearly, $\Aut{G}$ is a subgroup of $\SymmetryGroup{ V }$. In order to use automorphisms to produce a partition of $\mathcal{X}^N$, we shall let $\Aut{G}$ act on $\mathcal{X}^N$. We define the following group action (a map from $ \Aut{G}\times \mathcal{X}^N $ to $ \mathcal{X}^N$):
\begin{align}
  \label{lump:eq:group_action}
  f\cdot  x = y \in \mathcal{X}^N \iff x_{f(i) } = y_{ i }\, \forall i \in \setN{N} \, \text{ for } f  \in \Aut{G}, x \in \mathcal{X}^N  \eqstop
\end{align}
The rationale is that, for our purpose, an automorphism needs to preserve the local states of vertices as well. Note that the action of the group $\Aut{G}$ defined above can be used to introduce an equivalence relation on $ \mathcal{X}^N$ as follows: we say $x$ and $y $ are equivalent with respect to the action of $\Aut{G}$, denoted as $x\sim y$, if and only if there exists an $f \in \Aut{G}$ such that $  f\cdot  x = y $. The equivalence classes $\{ \tilde{\mathcal{X}}_1, \tilde{\mathcal{X}}_2, \ldots, \tilde{\mathcal{X}}_M     \}$   of the  relation $\sim$ yield a lumpable partition of $\mathcal{X}^N$.  Moreover, the partition thus obtained is finer than $  \{    \mathcal{X}_a \mid a \in \diophantine{N}{K}  \}$.  We prove this in the following.

\begin{myProposition}
  The partition $\{ \tilde{\mathcal{X}}_1, \tilde{\mathcal{X}}_2, \ldots, \tilde{\mathcal{X}}_M     \}$  induced by the equivalence relation $\sim$, \ie, the quotient space $ \mathcal{X}^N/\sim  $,   is a refinement of $  \{    \mathcal{X}_a \mid a \in \diophantine{N}{K}  \}$. That is, for each $i \in \setN{M}$, there exists an $a \in \diophantine{N}{K}$ such that $  \tilde{\mathcal{X}}_{i}  \subseteq \mathcal{X}_{a} $.
  \label{lump:propostion:refinement}
\end{myProposition}
\begin{proof}
Pick any $  \tilde{\mathcal{X}}_{i}   $ and  $ x \in  \tilde{\mathcal{X}}_{i}   $. Then, $a=c(x) \in \diophantine{N}{K}$, and therefore, $x \in \mathcal{X}_{a}$. The proof completes when we show that every other $y$ in  $  \tilde{\mathcal{X}}_{i}   $ is also in $\mathcal{X}_{a}$. Now, $y \in  \tilde{\mathcal{X}}_{i}   $ implies $x \sim y$, and therefore, there exists an $f \in   \Aut{G}$ such that $f\cdot x=y$. From the permutation invariance of $c$, we get $c(y)= c( f\cdot x  ) = c(x)=a$ implying $y \in \mathcal{X}_{a}$.
\end{proof}

\begin{myTheorem}
    The  CTMC $X$ with transition rate matrix~$Q$   (or equivalently the linear system  $  \dot{p}=p Q$) is lumpable with respect to the quotient space $ \mathcal{X}^N/\sim  $,  the partition $\{ \tilde{\mathcal{X}}_1, \tilde{\mathcal{X}}_2, \ldots, \tilde{\mathcal{X}}_M     \}$  induced by the equivalence relation $\sim$.
    \label{lump:proposition:automorphism_lumpability}
\end{myTheorem}

Before proving Theorem~\ref{lump:proposition:automorphism_lumpability}, we prove the following useful lemma regarding the neighbourhood function and the action of the group $\Aut{G}$.

\begin{myLemma}
  For all $i \in \setN{N}$ and for any $z \in \mathcal{X}^N$, the following is true for all $f \in \Aut{G}$: 
  \begin{align}
    n_{f^{-1}(i)}(f \cdot z) = n_i(z) \eqstop
    \label{lump:eq:neighbourhood_invariance}
  \end{align}
  \label{lump:lemma:neighbourhood}
\end{myLemma}
\begin{proof}[Proof of Lemma~\ref{lump:lemma:neighbourhood}]
Let us put $  f \cdot z=x$ and $f^{-1}(i)=k $.  If $d_k=0$, the assertion follows immediately because both sides of \autoref{lump:eq:neighbourhood_invariance}  are the empty set. Therefore, we assume $d_k=l>0$.  Then,
\begin{align*}
  n_{f^{-1}(i)}(f\cdot z) ={} & ( x_{i_1}, x_{i_2}, \ldots, x_{i_l}  ) \text{ if } (i_j, k) \in E\; \forall j \in \setN{i_l} \\
  ={} & ( z_{f(i_1)}, z_{f(i_2)}, \ldots, z_{f(i_l)}  )  \text{ if } (i_j, k) \in E\; \forall j \in \setN{i_l} \\
  ={} & n_{f(k)}(z) \eqcomma
\end{align*}
but $f(k)=i$ implying $  n_{f^{-1}(i)}(f\cdot z)  = n_i (z)$.
\end{proof}

Now we present the proof of Theorem~\ref{lump:proposition:automorphism_lumpability}.
\begin{proof}[Proof of Theorem~\ref{lump:proposition:automorphism_lumpability}]
  We check the  Dynkin's criterion to establish lumpability. For any two distinct  $i,j \in \setN{M}$, we check if $  \tilde{q}_{i,j}=  \sum_{y    \in  \tilde{\mathcal{X}}_j  }  q_{x,y}= \sum_{y    \in  \tilde{\mathcal{X}}_j  }  q_{z,y}   $ for each distinct pair  $x, z   \in  \tilde{\mathcal{X}}_i $. Since $z\sim x$, there exists an $f   \in \Aut{G}$ such that $f\cdot z=x$. The idea is to apply $f$ on the states of $ \tilde{\mathcal{X}}_j   $ and then show that, for any two states    $x, z \in  \tilde{\mathcal{X}}_i $, there are two states $y, f\cdot y \in  \tilde{\mathcal{X}}_j  $ such that the neighbourhood information are preserved. 
\begin{align*}
 \sum_{y    \in  \tilde{\mathcal{X}}_j  }  q_{x,y} = {}&  \sum_{y   \in  \tilde{\mathcal{X}}_j  } \sum_{i \in \setN{N}}  \indicator{ x_i \neq y_i, x_j=y_j \, \forall j \neq i   }   \gamma  (x_i, y_i, c(n_i (x))    ) \\
   ={} &   \sum_{f\cdot   y   \in  \tilde{\mathcal{X}}_j  } \sum_{i \in \setN{N}}  \indicator{ x_i \neq y_{f(i)}, x_j=y_{f(j)}  \, \forall j \neq i   }   \gamma  (x_i, y_{f(i)}, c(n_i (x))    ) \\
    ={} &   \sum_{ f\cdot   y   \in  \tilde{\mathcal{X}}_j  } \sum_{i \in \setN{N}}  \indicator{ z_{f(i)}\neq y_{f(i)}, z_{f(j)}=y_{f(j)}  \, \forall j \neq i   }   \gamma  (z_{f(i)}, y_{f(i)}, c(n_i ( f\cdot z ))    ) \\
    ={} &   \sum_{   f\cdot   y   \in  \tilde{\mathcal{X}}_j  } \sum_{ f^{-1}( i) \in \setN{N}}  \indicator{ z_{i}\neq y_{i}, z_{j}=y_{j}  \, \forall j \neq i   }   \gamma  (z_{i}, y_{i}, c(n_i (  z ))    )  \\
  ={ } &  \sum_{   y   \in  \tilde{\mathcal{X}}_j  } \sum_{i \in \setN{N}}  \indicator{ z_{i}\neq y_{i}, z_{j}=y_{j}  \, \forall j \neq i   }   \gamma  (z_{i}, y_{i}, c(n_i (  z ))    )
  ={}   \sum_{y    \in  \tilde{\mathcal{X}}_j  }  q_{z,y}  \eqcomma
\end{align*}
where  we have used $   n_{f^{-1}(i) }( f \cdot z ) = n_i (z)$ from \autoref{lump:lemma:neighbourhood}. Denoting  common value by $  \tilde{q}_{i,j}=  \sum_{y    \in  \tilde{\mathcal{X}}_j  }  q_{x,y} $, the matrix  $\tilde{Q} = ((  \tilde{q}_{i,j} ))$ is   the transition rate matrix of $\aggregation{X}$. 
\end{proof}

\begin{myRemark}
  From the perspective of group theory, finding the lumping classes is equivalent to determining the orbits of  states  in $\mathcal{X}^N$ with respect to the group $\Aut{G}$. For a state $x \in \mathcal{X}^N$, the orbit of $x$ with respect to the action of the group $\Aut{G}$, denoted as $\Aut{G} \cdot x$,  is defined by $\Aut{G}\cdot x = \{ f\cdot x \mid f \in \Aut{G}  \}$.
\end{myRemark}


%
%
%

\begin{myExample}[Complete graph]
The automorphism group $\Aut{G}$ for the complete graph is $\SymmetryGroup{\setN{N}}$. Therefore, any two states $x,y \in \mathcal{X}^N$ can be lumped together if $y$ is a rearrangement of components of $x$, \ie, $y= f\cdot x$ for some $f \in \SymmetryGroup{\setN{N}}$. As a consequence, $\{    \mathcal{X}_a \mid a \in \diophantine{N}{K}  \}$ itself is a lumpable partition of $\mathcal{X}^N$.
\end{myExample}

\begin{myExample}[Star graph]
An automorphism on a star graph  leaves the central node (root) unchanged and  permutes the rest of the nodes (leaf nodes) in any possible manner. Without loss of generality, let us assume the central node is labelled $N$. Then, the automorphism group $\Aut{G}$ is given by $ \Aut{G} =\{g \in \SymmetryGroup{\setN{N}} \mid g(N)= N, g(i)=f(i) \forall i\in \setN{N-1} \text{ for some }  f \in \SymmetryGroup{\setN{N-1}} \} $.
\end{myExample}

\begin{myExample}[Cycle graph]
The automorphisms of a cycle graphs are the reflections and rotations of the graph, forming a group that is also known as the dihedral group. Therefore, there are $2N$ automorphisms. In \cite{Simon2011Exact}, the authors show that  the dihedral group leads to a non-trivial lumping of states.
\end{myExample}

\begin{myExample}[Trees]
For a star graph, we noted that an automorphism  permutes the leaves but  needs to leave the root unchanged. Similarly, for a tree, we start with the leaves.  Any two leaves connected to the same parent node can be freely permuted. However, whenever we permute two leaf nodes that have different parents, we also need to permute the parents to preserve the neighbourhood structure. Therefore, an  automorphism on a tree necessarily maps vertices to vertices at the same height.
\end{myExample}


%
%
%
%
%


\setcounter{equation}{0}
\section{Lumping states using local symmetry}
\label{lump:sec:localSymmetry}

In this section, we discuss lumping ideas based on a local notion of automorphism. In many cases, the number of automorphisms decrease drastically as the graph grows arbitrarily large. For instance, it is known that  Erd\"{o}s-R\'{e}nyi random graphs tend to be asymmetric with probability approaching unity as the size of the graph~$N$ grows to infinity \cite{luczak1988automorphism}. Similar statements are true for $d$-regular random graphs under various sets of conditions on $d$ relative to the number of vertices~$N$ \cite{Kim2002RRG}, and random graphs with specified degree distributions \cite{McKay1984Auto}.  As a consequence, the automorphism-based lumping tends to be ineffective in state space reduction as the size of the graphs grows arbitrarily. Therefore, it is desirable to bring in a notion of local automorphism or local symmetry that would allow swapping vertices that are locally indistinguishable (\ie, have similar neighbourhoods), but are not  so globally. This notion of symmetry is  weaker than an automorphism, which endows global symmetry on a graph. However, the potential gain is in the ability to engender state space reduction when the graph grows arbitrarily large rendering automorphism-based lumping virtually ineffective. In the following, we make these ideas precise.

\subsection{Local symmetry}
There have been several attempts to formulate a more flexible notion of local symmetry. However, the literature seems divided on this and there is not a single  universally accepted concept. In our setup, it seems intuitive that two vertices that are locally indistinguishable in a large graph would also behave indistinguishably, and therefore, can be swapped. A notion of local symmetry identifying  such vertices was proposed in \cite{Barbosa2016LocalSymmetry}, which we adopt in this paper. We need a few definitions to make precise what we mean by two vertices being locally indistinguishable.


In order to define locality, we need some notion of distance between vertices of $G$. Let $d(u,v)$ denote the smallest distance (length of the minimal path) between two vertices $u,v \in V$. If $u$, and $v$ are not connected, \ie, there is no path between them, we simply set $d(u,v)=\infty$. 

\begin{myDefinition}
  Given a vertex $u \subseteq V$, define its $k$-neighbourhood in $G$, denoted by $N_{k}(u)$,  as follows:
  \begin{align}
    N_{k} : V \longrightarrow 2^V \text{ such that } N_{k}(u)\defeq \{ v \in V \mid d(u,v) \leq k\} \eqstop
  \end{align}
\end{myDefinition}
%
Let $G[   N_{k}(u)]$ denote the  subgraph of $G$ induced by $ N_{k}(u) $.
The notion of locality we adopt in this paper hinges on these $k$-neighbourhoods and their induced subgraphs. If two vertices induce isomorphic subgraphs, they are  indistinguishable locally and  we say they are $k$-locally symmetric \cite{Barbosa2016LocalSymmetry}.
\begin{myDefinition}
Two vertices $u,v\in V$ are defined to be $k$-locally symmetric if there exists an isomorphism $f$ between $G[N_k(u)]$ and $G[N_k(v)]$ such that $ f(u)=v $.
  \label{lump:definition:k_localSymmetry}
\end{myDefinition}
Therefore, two vertices $u,v \in V$ are $k$-locally symmetric if their $k$-th order local structures ($k$-hop neighbourhoods) are equivalent in the sense that there is a structure-preserving (edge-preserving in this case) bijection between them.  When $k=1$, we simply say the vertices are \emph{locally} symmetric.

As with automorphism, local symmetries also induce an equivalence relation on the set of vertices $V$. We say two vertices $u,v\in V$ are equivalent with respect to $k$-local symmetry, denoted by $u \overset{k} \sim v$,  if there exists an isomorphism $f$ between $G[N_k(u)]$ and $G[N_k(v)]$ such that $f(u)=v$. The notion of local symmetry is related to the concept of views in discrete mathematics literature \cite{Hendrickx2014Views,YamashitaKameda96}.  The view of depth $k$ of a vertex is a tree containing all walks of length $k$ leaving that vertex.  However, please note that, in our context, the induced subgraphs $G[N_k(u)]$ need not be trees. 
The following facts about local symmetry are useful for our study of lumpability \cite{NORRIS1995Covers,Barbosa2016LocalSymmetry}.
\begin{myProposition}
  \label{lump:proposition:local_symmetry_properties}
  The following properties are satisfied by $k$-local symmetry:
  \begin{enumerate}[label=\color{tud3d}{P\arabic*}]
    \item \label{lump:proposition:local_symmetry_properties:item1} For  $u,v \in V$, $ u \overset{k+1} \sim v  \implies u \overset{k} \sim v   $. Consequently, ${V}/\overset{k+1} \sim$ , the equivalence classes of $\overset{k+1} \sim$   form a refinement of ${V}/\overset{k} \sim$, the equivalence classes of $\overset{k} \sim$.
    \item \label{lump:proposition:local_symmetry_properties:item2} If the equivalence classes of $ \overset{k+1} \sim$ are the same as those of $ \overset{k} \sim$, the equivalence classes of \emph{all} $ \overset{k+j} \sim$ are the same as those of $ \overset{k} \sim$, for $j \in \setOfNaturals$.
    \item \label{lump:proposition:local_symmetry_properties:item3} If $k \geq \mathrm{diam}(G)$, the diameter of $G$, then, for two vertices $u,v \in V$, we have  $ u  \overset{k} \sim v  \iff  $ there exists an $f \in \Aut{G}$ such that $f(u)=v$. That is, $k$-local symmetry is equivalent to automorphism if $k$ is as large as the diameter of $G$.
  \end{enumerate}
\end{myProposition}

In addition to the above, it can be verified  that the quotient spaces ${V}/\overset{k} \sim$ are equitable partitions \cite[Chapter 9]{godsil2013algebraic} for each $k\geq 1$. We use these properties to lump states of $\mathcal{X}^N$ in the next section.

\subsection{Lumping states using local symmetry}
The procedure to lump states in $\mathcal{X}^N$ using local symmetry is similar to the procedure used to lump states using automorphism. However, unlike the case with automorphism, we now allow permutations that only need to ensure symmetry locally. That is, in order to lump states using $k$-local symmetry, we allow permuting two vertices $u$ and $v$ in $V$ if and only if $u$ and $v$ are $k$-local symmetric. Therefore, define
\begin{align}
  \Psi_k (G) \defeq \{ f \in \SymmetryGroup{V} \mid f(u)=v \iff u \overset{k}\sim v, \text{ for } u, v \in V \} \eqstop
\end{align}
We refer to   $\cardinality{  \Psi_k (G)  }$ as the number of local symmetries. It can be verified that $  \Psi_k (G)$, for each $k \geq 1$, forms a group under the composition of maps. Therefore, we can let the group $  \Psi_k(G) $ act on $\mathcal{X}^N$. We define the action of  $  \Psi_k(G) $ as follows:
\begin{align}
  \label{lump:eq:local_group_action}
  f\cdot  x = y \in \mathcal{X}^N \iff x_{f(i) } = y_{ i }\, \forall i \in \setN{N} \, \text{ for } f  \in  \Psi_k (G) , x \in \mathcal{X}^N  \eqstop
\end{align}
Note that a state $x$ in $\mathcal{X}^N$ is taken to $y$ if and only if the local states of all vertices are preserved and two vertices are swapped only when they are $k$-local symmetric. The above action induces the following partition of the state space: two states $ x,y \in \mathcal{X}^N$ are said to be equivalent with respect to $k$-local symmetry, denoted as $x \overset{k}\sim y$,   if there exists an $f \in  \Psi_k(G)$ such that $f\cdot x=y$. We use the same symbol $\overset{k}\sim$ since  there is no scope of confusion. The equivalence classes of $\overset{k}\sim $ are obtained, as before, by determining  the orbits of states in $\mathcal{X}^N$. The orbit of a state $x \in \mathcal{X}^N$ is given by $  \Psi_k(G) \cdot x \defeq \{  f\cdot x \in \mathcal{X}^N \mid f  \in  \Psi_k (G)  \}  $.

The partition thus obtained (based on $k$-local symmetry) does not, in general, guarantee lumpability, \ie, $X$ need to be lumpable with respect to $\mathcal{X}^N/\overset{k}\sim$. We say $X$ is approximately lumpable with respect to this partition and seek to quantify the approximation error in the next section.    The following observation is integral to the quantification of the approximation error incurred when states of $\mathcal{X}^N$ are lumped according to $k$-local symmetry instead of automorphism.

\begin{myProposition}
  The quotient space $\mathcal{X}^N/\overset{k+1}\sim$ is a refinement of $\mathcal{X}^N/\overset{k}\sim$.
  \label{lump:proposition:refinement_local_symmetry}
\end{myProposition}
\begin{proof}[Proof of Proposition~\ref{lump:proposition:refinement_local_symmetry}]
  Let $ \mathcal{X}_{1}^{(k+1)} , \mathcal{X}_{2}^{(k+1)} ,\ldots,  \mathcal{X}_{M_{k+1}}^{(k+1)}  $ be the equivalence classes of $\overset{k+1} \sim$. Also, denote the equivalence classes of $\overset{k} \sim$ by  $ \mathcal{X}_{1}^{(k)} , \mathcal{X}_{2}^{(k)} ,\ldots,  \mathcal{X}_{M_k}^{(k)}  $. Let $ i \in  \setN{M_{k+1}}$ and $x \in  \mathcal{X}_{i}^{(k+1)} $. If $ \mathcal{X}_{i}^{(k+1)}  $ is singleton, identity map is the only map in $  \Psi_{k+1} (G)  $, but it is also in $\Psi_{k} (G)$. Therefore, $x \in  \mathcal{X}_{j}^{(k)} $ for some $j \in \setN{  M_k}$, and the assertion follows. If $  \mathcal{X}_{i}^{(k+1)} $ has at least two elements, say, $x, y$, then $y \overset{k+1}\sim x$. By Propostion~\ref{lump:proposition:local_symmetry_properties}, we must have $y \overset{k}\sim x$. Therefore, there exists a $j \in \setN{  M_k}$ such that $x,y \in   \mathcal{X}_{j}^{(k)}  $. Since the choice of $x,y$ is arbitrary, the assertion follows.
\end{proof}

For practical applications, one would start with $\mathcal{X}^N/ \overset{1}\sim$ and then iteratively obtain further refinements $\mathcal{X}^N/ \overset{2}\sim, \mathcal{X}^N/ \overset{3}\sim  $, and so on until satisfactory accuracy is achieved (assuming we can quantify accuracy for the time being). In the light of Proposition~\ref{lump:proposition:local_symmetry_properties}, two important remarks are in place. They emphasise the benefits of local symmetry-driven lumping over the automorphism-driven one.

\begin{myRemark}
In an algorithmic implementation, \autoref{lump:proposition:local_symmetry_properties:item2} in  Proposition~\ref{lump:proposition:local_symmetry_properties}  provides a stopping rule for an iterative procedure to obtain local symmetry-driven partitions. That is, we can stop at the first instance of no improvement (the equivalence classes of $\overset{k+1}\sim$ and $\overset{k}\sim$ are the same).
\end{myRemark}

\begin{myRemark}
The diameters in many random graphs grow  slowly as the number of vertices goes to infinity. For instance, the diameter of  Erd\"os-R\'enyi random graphs with $N$ vertices and edge probability $\lambda/N$, for some fixed $\lambda>1$, grows as $\log N$ \cite{riordan2010diameter}.
In the light of  \autoref{lump:proposition:local_symmetry_properties:item3} in  Proposition~\ref{lump:proposition:local_symmetry_properties}, our approach needs (at most) as many steps as the diameter of $G$ to produce an \emph{exactly} lumpable partition of $\mathcal{X}^N$. Note that  $k \geq   \mathrm{diam}(G)$ is only a sufficient condition for $\mathcal{X}^N/  \overset{k}\sim$ to be an exactly lumpable partition. For practical purposes, we may  achieve sufficient accuracy (including exact lumpability) even for small values of $k < \mathrm{diam}(G) $.
\end{myRemark}

Our local symmetry-driven lumping approach shares a close relationship with what are known as fibrations in algebraic graph theory. We briefly describe the relationship in the following.

\setcounter{equation}{0}

\section{Graph fibrations}
\label{lump:sec:fibrations}
Fibrations of graphs were first inspired by fibrations between a pair of categories \cite{boldi2002fibrations}. Although the idea of fibrations originated from category theory, it has deep implications for graph theory, theoretical computer science, and other mathematical disciplines. For instance, in \cite{boldi2006PageRank}, the authors discuss its interesting connections to PageRank citation ranking algorithm. The authors in \cite{NIJHOLT2016NetworkDynamics} explore the similarities between dynamical systems with a network structure and dynamical systems with symmetry by means of fibrations of graphs. 
Let us now define the necessary graph theoretic concepts.

\begin{figure}
  \centering
  \includegraphics[width=0.75\columnwidth]{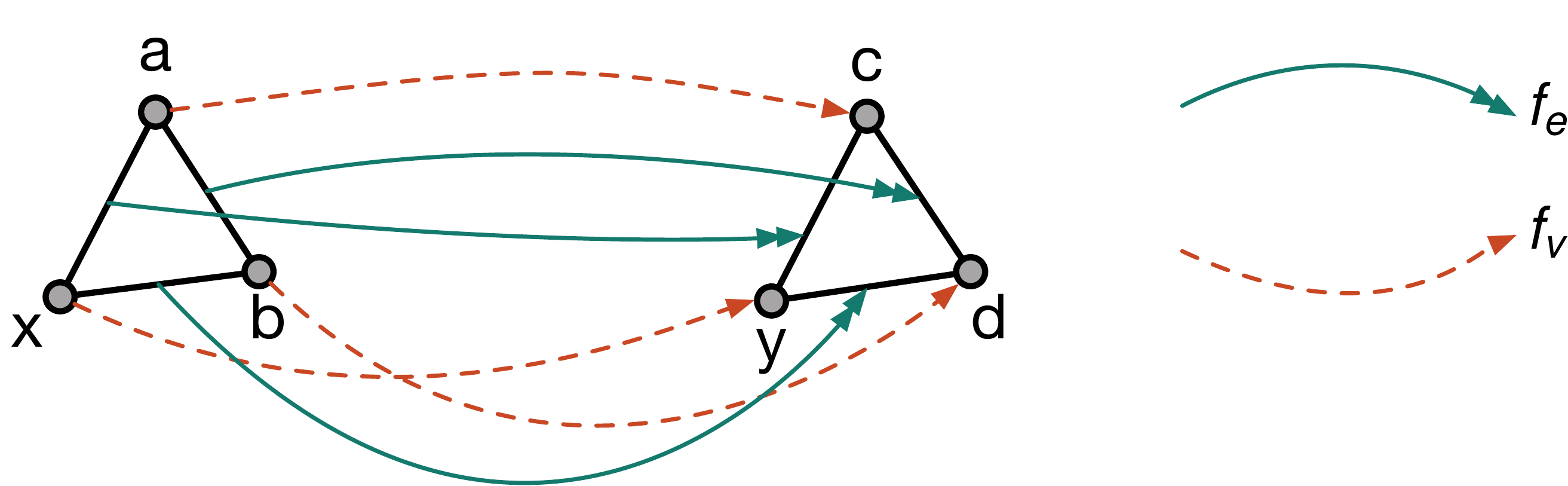}
  \caption{\label{lump:fig:fibration_localSymmetry}%
  Fibrations map   vertices to vertices and edges to edges. When three vertices form a traingle, fibrations also preserve the  triangle structure. Therefore, one can define an isomorphism between local neighbourhoods using fibrations.}
\end{figure}

Given the graph $G=(V,E)$, we first define the source and target maps $\mathsf{s}_G, \mathsf{t}_G : E \longrightarrow V$ on $G$ such that $\mathsf{s}_G (u,v) =u$ and $\mathsf{t}_G(u,v)=v$ for each $(u,v)\in E$.  Let $H =(V', E')$ be another graph. The source and the target maps $\mathsf{s}_H, \mathsf{t}_H$ are defined analogously. A map $f \defeq (f_v, f_e)$, where $f_v: V \longrightarrow V'$ and $f_e : E \longrightarrow E'$, is called a \emph{graph morphism} between $G$ and $H$ (from $G$ to $H$, to be precise) if $f_v$ and $f_e$ commute with the source and the target maps of $G$ and $H$, \ie, if $\mathsf{s}_H f_e = f_v \mathsf{s}_G  $ and $ \mathsf{t}_H f_e = f_v \mathsf{t}_G $. A morphism is called an \emph{epimorphism} if both $f_v$ and $f_e$ are surjective. Finally, we define a graph fibration as follows \cite{boldi2002fibrations}:

\begin{myDefinition}
\label{lump:defn:fibrations}
A morphism $f \defeq (f_v, f_e)$ between two graphs $G=(V,E)$ and $H=(V',E')$   is called a {fibration} between graphs $G$ and $H$  (from $G$ to $H$, to be precise) if, for each edge $a\in E'$ and for each $x \in V$ satisfying $f_v(x)= \mathsf{t}_H (a)$, there exists  a unique edge $a_x \in E$ such that $f_e(a_x)=a   $ and $\mathsf{t}_G(a_x)=x$. The edge $a_x$ thus found is called the lifting of $a$ at $x$, and is denoted by $f_e^{-1}(a)$. The graph $G$ is then called {fibred over} $H$. The fibre over a vertex $y \in V'$, denoted by $\fibre{y}$,  is the set of vertices in $V$ that are mapped to $y$, \ie, $\fibre{y} \defeq \{ x\in V \mid f_v(x)=y \} $.
\end{myDefinition}

%
%

In the original paper \cite{boldi2002fibrations}, the authors define colour preserving graph morphisms when graphs are endowed with a colouring function. In that case, $f_e$ also commutes with the colouring function. For our present purposes, we do not require  this generality and only consider uncoloured graphs. In \cite{boldi2002fibrations}, the authors showed that a left action of a group on $G$ can be used to induce fibrations. They also show that fibrations and epimorphisms satisfying certain local in-isopmorphism property are equivalent  \cite[Theorem 2]{boldi2002fibrations}.    Indeed, fibrations have a close relationship with the notion of local symmetry described in \autoref{lump:sec:localSymmetry}.  The proof of the following proposition follows analogously from  \cite[Theorem 2]{boldi2002fibrations}. However, for the sake of completeness, we also provide it in \autoref{lump:sec:appendixA}.

\begin{myProposition}\label{lump:proposition:fibration_1}
Let $f\defeq (f_v,f_e)$ be a fibration of the graph $G=(V,E)$, \ie, a fibration from $G$ to $G$ itself. Pick two vertices $x,y \in V$. If $x \in \fibre{y}$,  the vertices $x,y$ are locally symmetric,\ie, $x \overset{1} \sim y $. Moreover, if the vertices $x,y$ are locally symmetric, there exists a fibration such that $ x \in \fibre{y}$.
\end{myProposition}

The above proposition  essentially shows that the equivalence classes of local symmetry (with $k=1$) and fibres induced by a graph fibration are the same. Therefore, the fibres can also be used to aggregate the states of $\mathcal{X}^N$ to achieve approximate lumpability in the same fashion as we did with local symmetry.

\setcounter{equation}{0}
\section{Approximation error}
\label{lump:sec:approximation_error}
As the lumping based on local symmetry does not ensure Markovianness of the lumped process, we need to quantify the approximation error. In order to do so, we work with the uniformization of $X$.    Then we lump  $\uniformization{X}$ to produce $\aggregation{\uniformization{X}}$  according to $k$-local symmetry.   A direct assessment of the quality of aggregation is cumbersome. Therefore, it is suggested \cite{Geiger2015OptimalKL,Deng2011KL} that we \emph{lift} the aggregated process $\aggregation{\uniformization{X}}$ to a Markov chain on the same state space $\mathcal{X}^N$ as $ \uniformization{X}$ and then compare their transition probability matrices. The lifting allows us to use known metrics of divergence such as the Kullback-Leibler  divergence to quantify the approximation error. We follow the scheme depicted in \autoref{lump:fig:lifting}.

%
%

%

In order to fix ideas, let us lump $ \uniformization{X}$ according to $k$-local symmetry, \ie, according to the partition $\{ \mathcal{X}_{1}^{(k)} , \mathcal{X}_{2}^{(k)} ,\ldots,  \mathcal{X}_{M_k}^{(k)} \}  $ of $\mathcal{X}^N$ obtained as the equivalence classes of $\overset{k}\sim$. We introduce two notations in this connection. Let $\eta_k : \mathcal{X}^N \longrightarrow \setN{M_k}  $
be the partition function associated with $\overset{k}\sim$, \ie, $\eta_k(x)\defeq i \iff x \in  \mathcal{X}_{i}^{(k)}  $. For $u\in \mathcal{X}^N$, let us denote the equivalence class containing $u$ by $\coset{x}{k}$, \ie, $\coset{x}{k}\defeq  \mathcal{X}_{i}^{(k)} \iff x \in  \mathcal{X}_{i}^{(k)}  $. Note that, $\coset{x}{k} = \eta_k^{-1}(\eta_k (x))   $.

Let $T =((t_{i,j}))$ be the transition probability matrix associated with $ \uniformization{X}$. Now, since $X$ is not necessarily lumpable with respect to the partition $\{ \mathcal{X}_{1}^{(k)} , \mathcal{X}_{2}^{(k)} ,\ldots,  \mathcal{X}_{M_k}^{(k)} \}  $, for $i \neq j \in \setN{M_k} $ and two distinct $x,y \in \mathcal{X}_{i}^{(k)} $, the quantity $\sum_{z \in \mathcal{X}_{j}^{(k)}} t_{x,z}$ may not equal $\sum_{z \in \mathcal{X}_{j}^{(k)}} t_{y,z}$. If $\uniformization{X}$ is stationary with distribution $\pi$,  \ie, if $\pi$ is the solution to $\pi T= \pi$ and $p(0)=\pi$,  a natural estimate of the transition probability of the lumped process is the following
\begin{align}
\label{lump:eq:lumpedTransitionProb}
   \tilde{t}_{i,j}^{(k)} \defeq \frac{  \sum_{u \in  \mathcal{X}_{i}^{(k)}     }  \pi_u   \sum_{v \in  \mathcal{X}_{j}^{(k)} }  t_{u,v}  }{ \sum_{u \in  \mathcal{X}_{i}^{(k)} }   \pi_u  } \, \eqcomma \text{ for } i, j \in \setN{M_k} \eqstop
\end{align}
That is, we estimate the transition probabilities of the lumped process  $\aggregation{ \uniformization{X}}$    by averaging the different values $\sum_{z \in \mathcal{X}_{j}^{(k)}} t_{x,z}$ and $\sum_{z \in \mathcal{X}_{j}^{(k)}} t_{y,z}$, weighted by the stationary probabilities \cite{Geiger2015OptimalKL}. Let $\tilde{T}^{(k)} \defeq ((   \tilde{t}_{i,j}^{(k)}  ))$.  Now, we describe how the transition probabilities of the lifted Markov chain are calculated. There are two common ways of lifting $\aggregation{\uniformization{X}}$ to a Markov chain on $\mathcal{X}^N$; one using a probability vector, called $\pi$-lifting, and the other  using  the transition probabilities, called $P$-lifting.  Let us discuss $\pi$-lifting first.

%


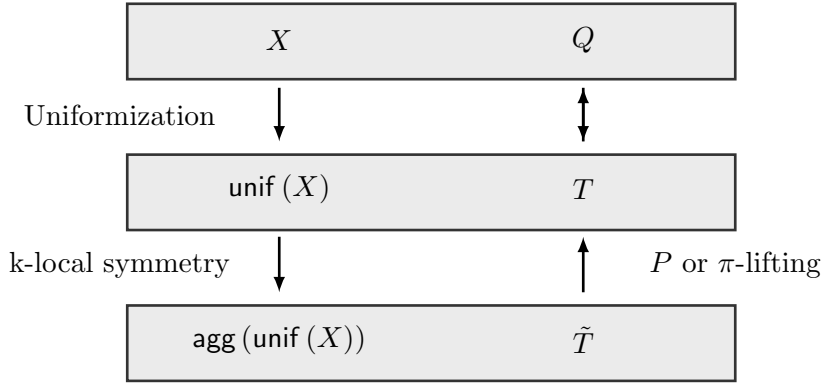
\begin{figure}
 \centering
 \normalfont
 \begin{tikzpicture} [inner sep=0cm, minimum size = 1.25cm]
   \draw[line width=1pt, draw=tud0d, fill=tud0a!40] (0,0) rectangle  +(8,1);
   \node(node1)[below] at (2,1.15){$X$};
   \node(node2)[below] at (6,1.15){$Q$};
   \draw[line width=1pt, draw=tud0d, fill=tud0a!40] (0,-2) rectangle  +(8,1);
   \node(node3)[below] at (2,-0.825){$\uniformization{X}$};
   \node(node4)[below] at (6,-0.825){$T$};
   \draw[line width=1pt, draw=tud0d, fill=tud0a!40] (0,-4) rectangle  +(8,1);
   \node(node5)[below] at (2,-2.825){$\aggregation{\uniformization{X}}$};
   \node(node6)[below] at (6,-2.825){$\tilde{T}$};
   \draw [-latex, line width=1pt] (node1) to node [minimum size=0.3cm, xshift=-2.1cm, yshift=0.0cm] {Uniformization} (node3);
   \draw [-latex, line width=1pt] (node3) to node [minimum size=0.3cm, xshift=-2.1cm, yshift=0.0cm] {k-local symmetry} (node5);
   \draw [-latex, line width=1pt] (node6) to node [minimum size=0.3cm, xshift=2.0cm, yshift=0.0cm] {$P$ or $\pi$-lifting} (node4);
   \draw [-latex, line width=1pt] (node4) to node [minimum size=0.3cm, xshift=2.0cm, yshift=0.0cm] {} (node2);
   \draw [-latex, line width=1pt] (node2) to node [minimum size=0.3cm, xshift=2.0cm, yshift=0.0cm] {} (node4);
 \end{tikzpicture}
 \caption{\label{lump:fig:lifting} %
 Lifting procedure used to assess the quality of the approximation.}
\end{figure}

\begin{myDefinition}[$\pi$-lifting]
	The $\pi$-lifting of $\eta_k(\uniformization{X}  )$ is a DTMC with transition probability matrix $T^{\pi}_{k} \defeq ((t_{u,v}^{\pi,k}    ))$ where
	\begin{align}
	t_{u,v}^{\pi,k}  \defeq \frac{   \pi_v  }{   \sum_{  x \in \coset{v}{k}  }  \pi_x   }  \tilde{t}_{ \eta_k(u)    ,\eta_k (v)}^{(k)} \, \eqcomma \text{ where } u, v \in \mathcal{X}^N \eqstop
	\label{lump:eq:pi_lifting}
	\end{align}
  \label{lump:Lump:defn:pi_lifting}
\end{myDefinition}

Please note that, in principle,  $\pi$-lifting can be done using any probability vector as long as the denominator remains non-zero for the choice of the candidate probability vector.  Nevertheless, the most common choice is the stationary probability vector. The reason for this choice is the fact that the stationary probability vector achieves the minimum KL divergence rate \cite{Deng2011KL}. For this reason, we consider $\pi$-lifting with the stationary distribution for numerical computations in this paper. Another immediate consequence of $\pi$-lifting is that the lifted Markov chain with transition probability matrix $T_k^{\pi}$ given in Definition~\ref{lump:Lump:defn:pi_lifting} is lumpable with respect to the partition $\mathcal{X}^N/\overset{k}\sim$ and has $\pi$ as the stationary probability. Now, we define the approximation error.

\begin{myDefinition}
We define the approximation error to be the KL divergence rate  between $\uniformization{X}$ and 
the lifted DTMCs. Therefore, for  $\pi$-lifting, the approximation error is given by
\begin{align}
\KL{T}{ T^{\pi}_{k}  } \defeq \sum_{u  \in \mathcal{X}^N  } \sum_{v \in \mathcal{X}^N }  \pi_u t_{u,v} \log \left(  \frac{ t_{u,v}  }{   t_{u,v}^{\pi,k}    }  \right) \eqstop
\end{align}
\end{myDefinition}
Having defined the approximation error, we show that it indeed decreases monotonically with the order of local symmetry. This is precisely the assertion of Theorem~\ref{lump:theorem:KL_monotonicity}. However, in order to prove Theorem~\ref{lump:theorem:KL_monotonicity}, we need to make use of the  following calculation, which we present in the form of a  lemma.

\begin{myLemma}
  For any two states $u,v \in \mathcal{X}^N$, and for any $k$, define the ratio
  \begin{align}
\rho_k (u,v) \defeq   \frac{ \sum_{t \in \coset{v}{k}  }    \pi_t      }{ \tilde{t}_{\eta_k(u), \eta_k(v)    }^{(k)}      }   = \frac{    \sum_{p \in \coset{u}{k}   }  \pi_p \sum_{q \in \coset{v}{k}  }    \pi_q     }{     \sum_{p \in \coset{u}{k}   }  \sum_{q \in \coset{v}{k}  }    \pi_p t_{p,q}  }  \eqstop
  \end{align}
  Then, the following recursion relation holds true:
\begin{align}
 \sum_{x \in \coset{u}{k}   }  \sum_{y \in \coset{v}{k}   }  \pi_x t_{x,y}  \rho_{k+1}(x,y)  ={} & \rho_k(u,v) \sum_{x \in \coset{u}{k}   }  \sum_{y \in \coset{v}{k}   } \pi_x t_{x,y}   \eqstop
\end{align}
  \label{lump:lemma:k_average}
\end{myLemma}
\begin{proof}[Proof of Lemma~\ref{lump:lemma:k_average}]
  By the refinement property of local symmetry  in   Proposition~\ref{lump:proposition:refinement_local_symmetry}, we can find distinct   integers $i_1,i_2,\ldots, i_m$ and $j_1,j_2, \ldots, j_n$ in $\setN{M_{k+1}}$ such that 
  \begin{align}
  	\label{lump:eq:u_k_partition}
    \coset{u}{k} =\cup_{ l \in \setN{m}  }    \mathcal{X}_{ i_l  }^{(k+1)} \text{ and }   \coset{v}{k} =\cup_{ l \in \setN{n}  }    \mathcal{X}_{ j_l  }^{(k+1)} \eqstop
  \end{align}
Therefore, we can split the summation over $  \coset{u}{k}$, $  \coset{v}{k}$ into disjoint equivalence classes of $\overset{k+1}\sim$. Within each of these equivalence classes of $\overset{k+1}\sim$, the quantity $\tilde{t}_{ \eta_{k+1}(x), \eta_{k+1}(y)  }$ is constant, and therefore, can be pulled out of the summation.  Therefore, 
\begin{align*}
  &\sum_{x \in \coset{u}{k}   }  \sum_{y \in \coset{v}{k}   }  \pi_x t_{x,y}  \rho_{k+1}(x,y) \\
    ={} & \sum_{p \in \setN{m}}  \sum_{q \in \setN{n}}   \sum_{x \in   \mathcal{X}_{ i_p  }^{(k+1)}   }  \sum_{y \in   \mathcal{X}_{ j_q  }^{(k+1)}   }  \pi_x t_{x,y} \left(    \frac{    \sum_{s \in \coset{x}{k+1}   }  \pi_s \sum_{t \in \coset{y}{k+1}  }    \pi_t      }{     \sum_{s \in \coset{x}{k+1}   }  \sum_{t \in \coset{y}{k+1}  }    \pi_s t_{s,t}  }     \right) \\
    ={} & \sum_{p \in \setN{m}}  \sum_{q \in \setN{n}}      \left(    \frac{    \sum_{s \in  \mathcal{X}_{ i_p  }^{(k+1)}  }  \pi_s \sum_{t \in  \mathcal{X}_{ j_q  }^{(k+1)}   }    \pi_t      }{     \sum_{s \in  \mathcal{X}_{ i_p  }^{(k+1)}  }  \sum_{t \in  \mathcal{X}_{ j_q  }^{(k+1)}  }    \pi_s t_{s,t}  }     \right)  \sum_{x \in   \mathcal{X}_{ i_p  }^{(k+1)}   }  \sum_{y \in   \mathcal{X}_{ j_q  }^{(k+1)}   }  \pi_x t_{x,y}\\
    ={} &  \sum_{x \in \coset{u}{k}   }  \sum_{y \in \coset{v}{k}   } \pi_x \pi_y
    ={} \rho_k(u,v)   \sum_{x \in \coset{u}{k}   }  \sum_{y \in \coset{v}{k}   } \pi_x t_{x,y}  \eqstop
\end{align*}
This completes the proof.
\end{proof}

Note that $\rho_k(u,v)=\rho_k(x,y)$ for any $u \overset{k}\sim x$ and $v \overset{k}\sim y$. Therefore, we can use the shorthand notation $\rho_k( \mathcal{X}_i^{(k)} ,  \mathcal{X}_j^{(k)}    ) $ to mean $\rho_k(u,v)$ for any $u \in \mathcal{X}_i^{(k)} ,  v \in \mathcal{X}_j^{(k)}  $.

\begin{myRemark}[Averaging argument]
The main implication of \autoref{lump:lemma:k_average} is that the quantity $\rho_k(u,v) $ can be seen as a weighted average of $\rho_{k+1} (x,y)$ where $x,y$'s are in the equivalence classes of $\overset{k+1}\sim$. The weights  are precisely 
\begin{align}
W_{ \coset{u}{k}, \coset{v}{k} } (  \mathcal{X}_{ i_p  }^{(k+1)} ,   \mathcal{X}_{ j_q  }^{(k+1)}   ) \defeq \frac{  \sum_{x \in   \mathcal{X}_{ i_p  }^{(k+1)}   }  \sum_{y \in   \mathcal{X}_{ j_q  }^{(k+1)}   }  \pi_x t_{x,y}  }{  \sum_{x \in \coset{u}{k}   }  \sum_{y \in \coset{v}{k}   } \pi_x t_{x,y}}  \eqcomma
\end{align}
where we have partitioned $\coset{u}{k}$ and $\coset{v}{k}$ into $ \mathcal{X}_{ i_p  }^{(k+1)} $'s and $ \mathcal{X}_{ j_q  }^{(k+1)} $'s respectively as shown in \autoref{lump:eq:u_k_partition}. We interpret  $W_{ \coset{u}{k}, \coset{v}{k} } (  \mathcal{X}_{ i_p  }^{(k+1)} ,   \mathcal{X}_{ j_q  }^{(k+1)}   ) $ as the weight for the cross-section $ \mathcal{X}_{ i_p  }^{(k+1)} \times  \mathcal{X}_{ j_q  }^{(k+1)}   $ with regards to the partition of $\coset{u}{k}$ and $\coset{y}{k}$ given in \autoref{lump:eq:u_k_partition}. Therefore, it follows from \autoref{lump:lemma:k_average} that
\begin{align}
\rho_k(\coset{u}{k} ,\coset{v}{k}) =  \sum_{p \in \setN{m}}  \sum_{q \in \setN{n}}   \rho_{k+1}(   \mathcal{X}_{ i_p  }^{(k+1)} ,   \mathcal{X}_{ j_q  }^{(k+1)}    ) W_{ \coset{u}{k}, \coset{v}{k} } (  \mathcal{X}_{ i_p  }^{(k+1)} ,   \mathcal{X}_{ j_q  }^{(k+1)}   )  \eqstop
\end{align}
Since the weights sum up to unity, $\rho_k(u,v) $ can be indeed seen as an average. Keeping this remark in mind, we now proceed to state and prove \autoref{lump:theorem:KL_monotonicity} about the monotonicity of the approximation error.
\label{lump:remark:averaging_argument}
\end{myRemark}

\begin{myTheorem}
  For $\pi$-lifting, the aggregation of states in $\mathcal{X}^N$ using  local symmetry ensures monotonically decreasing approximation error 
  with increasing order of local symmetry. That is,
  \begin{align}
    \KL{T}{T_{k+1}^{(\pi)} } \leq   \KL{T}{T_{k}^{(\pi)}} \, \text{ for all } k \geq 1 \eqstop
  \end{align}
\label{lump:theorem:KL_monotonicity}
\end{myTheorem}

\begin{proof}[Proof of \autoref{lump:theorem:KL_monotonicity}]
	By the refinement property of local symmetry proved in   \autoref{lump:proposition:refinement_local_symmetry}, we can partition $\setN{M_{k+1}} =\{1,2,\ldots,M_{k+1}\} $ into $\{\Lambda_1, \Lambda_2, \ldots, \Lambda_{M_k}   \}$ such that
	\begin{align*}
 \mathcal{X}_{i}^{(k)} = \cup_{l \in \Lambda_i }  \mathcal{X}_{l  }^{k+1}  \eqstop
	\end{align*}
Note that
\begin{align*}
 &\KL{T}{T_{k}^{(\pi)}}-   \KL{T}{T_{k+1}^{(\pi)} } \\
 ={} &  \sum_{i, j \in  \setN{M_k} }   \sum_{u  \in \mathcal{X}_{i}^{(k)}  } \sum_{v \in \mathcal{X}_{j}^{(k)}  }  \pi_u t_{u,v}   \log \left(   \frac{  \rho_k(u,v)        }{      \rho_{k+1}(u,v)     }     \right)  \\
 ={} &  \sum_{i, j \in  \setN{M_k} } (  \log \left( \rho_k(\mathcal{X}_{i}^{(k)},\mathcal{X}_{j}^{(k)}  )   \right)      \sum_{u  \in \mathcal{X}_{i}^{(k)}  } \sum_{v \in \mathcal{X}_{j}^{(k)}  }  \pi_u t_{u,v}   -     \sum_{u  \in \mathcal{X}_{i}^{(k)}  } \sum_{v \in \mathcal{X}_{j}^{(k)}  }  \pi_u t_{u,v}    \log \left( \rho_{k+1}(u,v)   \right)     ) \\
 ={} &  \sum_{i, j \in  \setN{M_k} }  \Theta_{i,j} \eqcomma
\end{align*}
where
\begin{align*}
 \Theta_{i,j} \defeq {}&  \log \left( \rho_k(\mathcal{X}_{i}^{(k)},\mathcal{X}_{j}^{(k)}  )   \right)      \sum_{u  \in \mathcal{X}_{i}^{(k)}  } \sum_{v \in \mathcal{X}_{j}^{(k)}  }  \pi_u t_{u,v} \\
 &{}- \sum_{p \in \Lambda_i}  \sum_{q \in \Lambda_j}  \sum_{u  \in \mathcal{X}_{p}^{(k+1)}  } \sum_{v \in \mathcal{X}_{q}^{(k+1)}  }  \pi_u t_{u,v}   \log \left( \rho_{k+1}(u,v)   \right)  \\
 ={} & ( \sum_{u  \in \mathcal{X}_{i}^{(k)}  } \sum_{v \in \mathcal{X}_{j}^{(k)}  }  \pi_u t_{u,v} ) \times  (  \log \left( \rho_k(\mathcal{X}_{i}^{(k)},\mathcal{X}_{j}^{(k)}  )   \right)  \\
 &{}  - \sum_{p \in \Lambda_i}  \sum_{q \in \Lambda_j}  W_{ \mathcal{X}_{i}^{(k)},\mathcal{X}_{j}^{(k)} }(\mathcal{X}_{p}^{(k+1)},\mathcal{X}_{q}^{(k+1)} )    \log \left( \rho_{k+1}(\mathcal{X}_{p}^{(k+1)},\mathcal{X}_{q}^{(k+1)}  )   \right)      ) \\
 \geq {} & 0 \eqcomma
\end{align*}
by Jensen's inequality and the averaging argument given in Remark~\ref{lump:remark:averaging_argument} and Lemma~\ref{lump:lemma:k_average}. 
This completes the proof.
\end{proof}


\begin{figure}[!th]
    \centering
    \includegraphics[width=\columnwidth]{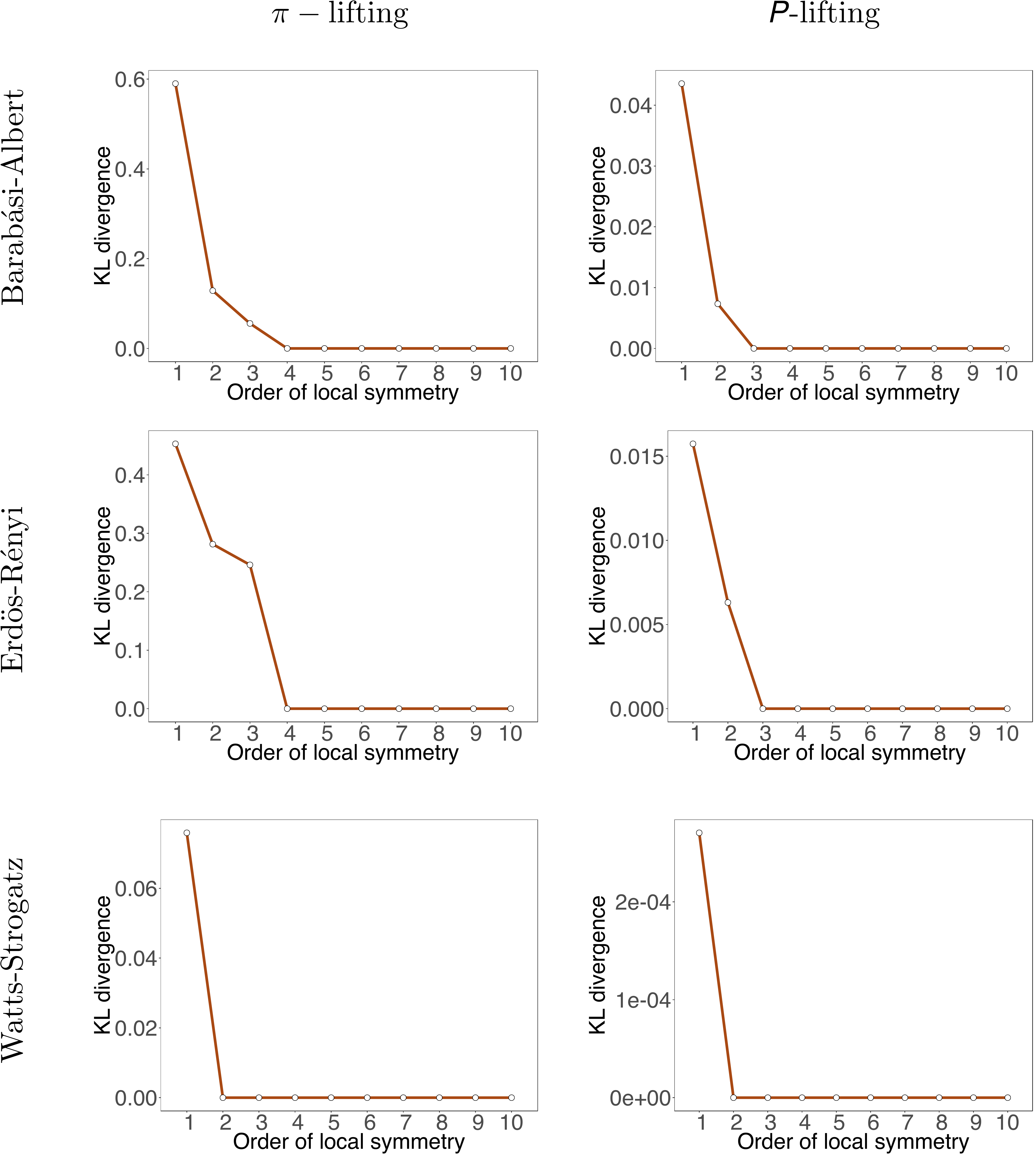}
    \caption[Monotonicity of KL divergence rate]{\label{lump:figure:KLDMonotonicity_Combined}%
      Monotonicity of the KL divergence with the order of the local symmetry for the SIS dynamics on different models of random graphs with $10$ vertices.  All graphs are undirected. The Erd\"os-R\'enyi graphs are created with edge probability $0.3$, while the Watts-Strogatz small world networks are created with rewiring probability~$0.3$ and each vertex having three neighbours. The infection and the recovery rates are both~$0.5$.}
\end{figure}

\begin{figure}[!th]
    \centering
    \includegraphics[width=\columnwidth]{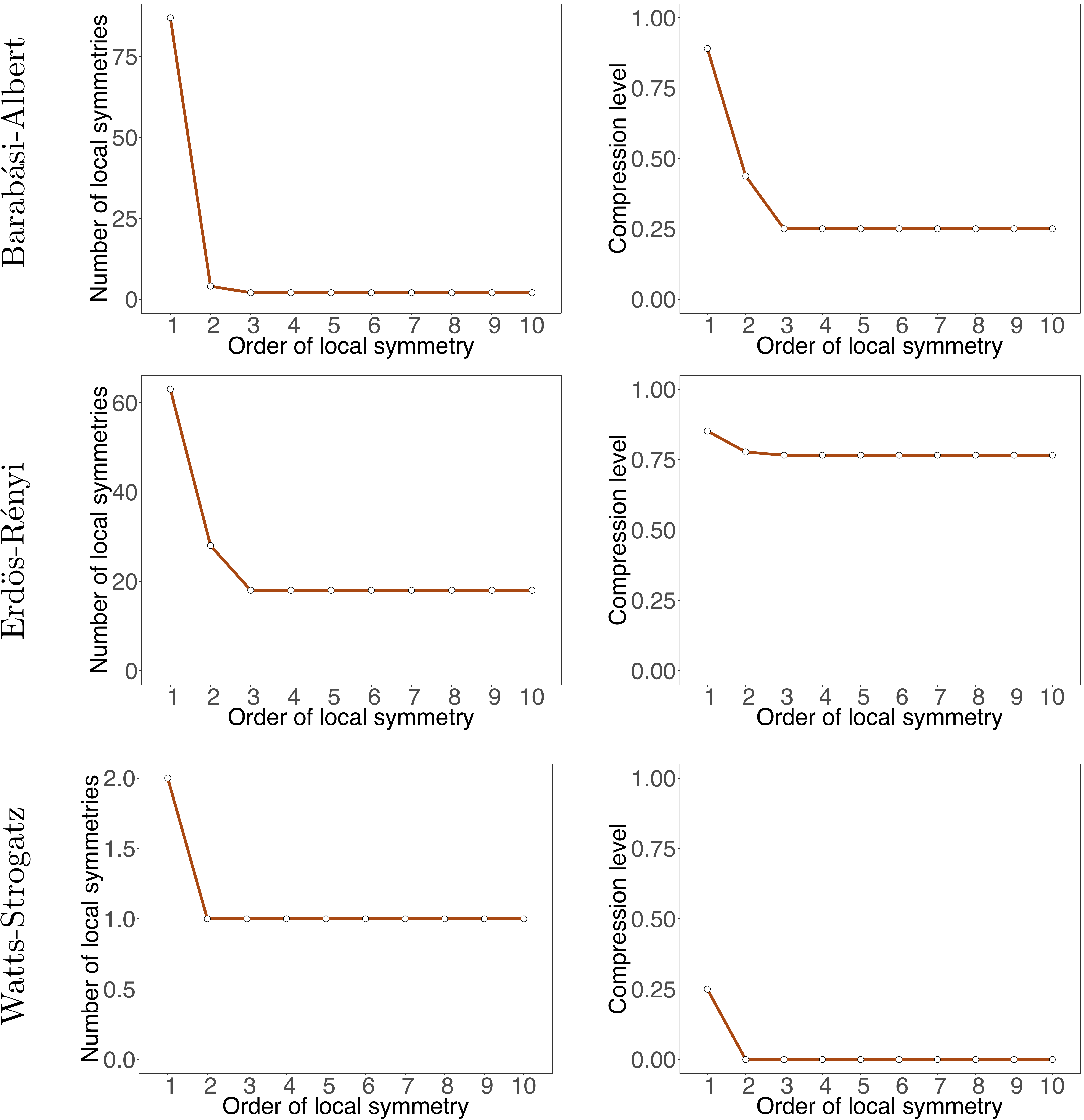}
    \caption[Compression level for local symmetry-based lumping]{\label{lump:lump:figure:Compression}%
      As we increase the order, the number of local symmetries (the cardinality of $\Psi_k$) decreases drastically. Therefore, the compression level  also decreases. The simulation set-up is the same as in \autoref{lump:figure:KLDMonotonicity_Combined}. 
      }
\end{figure}

Note that $\KL{T}{T_{k}^{(\pi)}}-   \KL{T}{T_{k+1}^{(\pi)} }=0$ is achieved if (and only if) equality is achieved in Jensen's inequality forcing the individual $\Theta_{i,j}$'s to be zeros. This is the case when the $\rho_k$ and $\rho_{k+1}$'s are equal. There are two possibilities. First, the equivalence classes of $\overset{k}{\sim}$ and $\overset{k+1}{\sim}$ are the same. In this case, by \autoref{lump:proposition:local_symmetry_properties}, the equivalence classes of all  $\overset{k+j}{\sim}$, for $j\geq 2$, will remain the same. Therefore, we have already reached automorphism, and hence, exact lumpability.   Second, the equivalence classes of $\overset{k}{\sim}$ and $\overset{k+1}{\sim}$ are different (so, we are not yet at automorphism), but exact  lumpability has already been achieved at order of local symmetry~$k$.  In both cases, we need not refine our partition further because exact lumpability has been achieved. Therefore,  $\KL{T}{T_{k}^{(\pi)}}-   \KL{T}{T_{k+1}^{(\pi)} }=0$ serves as a definite stopping rule for any iterative algorithmic implementation of local symmetry-driven lumping.

Now, we discuss the second type of lifting, which makes of the transition probabilities and is called $P$-lifting. The following is the definition.

\begin{myDefinition}[$P$-lifting]
	The $P$-lifting of $\eta_k(\uniformization{X}  )$ is a DTMC with transition probability matrix $T^{P}_{k} \defeq ((t_{u,v}^{P,k}    ))$ where, for $u, v \in \mathcal{X}^N$,
	\begin{align}
	t_{u,v}^{P,k}  \defeq \begin{cases}
	\frac{   t_{u,v}  }{   \sum_{  x \in \coset{v}{k}  }    t_{u,x}   }  \tilde{t}_{ \eta_k(u)    ,\eta_k (v)}^{(k)} \,  \text{ if   }     \sum_{  x \in \coset{v}{k}  }    t_{u,x} >0 \eqcomma \\
	\frac{1 }{  \cardinality{  \coset{v}{k}   }     }  \tilde{t}_{ \eta_k(u)    ,\eta_k (v)}^{(k)} \,    \text{ if }    \sum_{  x \in \coset{v}{k}  }    t_{u,x} =0 \eqstop
	\end{cases}
	\label{lump:eq:P_lifting}
	\end{align}
\end{myDefinition}

The approximation error for $P$-lifting is defined similarly. Note that $P$-lifting is sharp, in the sense that if the lumping is in fact exact, then $\KL{   T }{   T_{k}^{(P)} } =0$, whereas $\pi$-lifting is not \cite{Geiger2015OptimalKL}.  
In \autoref{lump:figure:KLDMonotonicity_Combined}, we show numerical results pertaining to Theorem~\ref{lump:theorem:KL_monotonicity}. We consider the Barab\'asi-Albert preferential attachment, the Erd\"os-R\'enyi and the Watts-Strogatz small-world random graphs. The claimed monotonicity is observed in all three cases. In fact, the KL divergence rate steeply decreases in all three cases, for both $\pi$ as well as $P$-lifting. The figures are particularly encouraging in that satisfactory level of accuracy is achieved even for small orders of local symmetry. Since one of the main purposes of aggregation is to engender state space reduction, we need to evaluate the performance of  local symmetry-driven aggregation in terms of some notion of compression level as well. Therefore, we define compression level ${C}$ at order of local symmetry~$k$ as follows:
\begin{align}
{C}(k)=1- \frac{M_k}{\cardinality{ \mathcal{X}^N   }  } \eqcomma
\end{align}
where $M_k$ is the cardinality of the quotient space $ \mathcal{X}^N/ \overset{k}\sim $, \ie, the number of equivalence classes of $\overset{k}\sim$. If there is no non-trivial local symmetries, the compression level is zero because the partition is simply $\{ \{x\} \mid x \in \mathcal{X}^N  \}$. In \autoref{lump:lump:figure:Compression}, we show how the number of local symmetries decreases drastically as we increase the order of local symmetry. Consequently, the compression level also falls steeply. This is expected because random graphs tend to become asymmetric as the number of vertices increases.

\begin{myRemark}
	Please note that \autoref{lump:theorem:KL_monotonicity} holds true for a general Markov chain whenever the partition function $\eta_{k+1}$ is a refinement of $\eta_k$. The fact that the partition functions $\eta_k,\eta_{k+1}$ are associated with the equivalence relations generated by $k$ and $k+1$-local symmetries is only sufficient for the validity of Theorem~\ref{lump:theorem:KL_monotonicity}, but not necessary. In fact, similar monotonicity can be proved, in similar fashion, even when $\eta_k,\eta_{k+1}$ are arbitrary partition functions defined on the state-space of a Markov chain such that $\eta_{k+1}$ is a refinement of $\eta_k$. Notably, such monotonicity can only be guaranteed for $\pi$-lifting. In \autoref{lump:sec:appendixA}, we provide a counterexample to establish that such monotonicity   fails for $P$-lifting when  arbitrary partition functions (one being a refinement of the other) are considered. However, this observation is about a general Markov chain. 
	For our MABM, we  observe similar monotonicity for $P$-lifting using numerical computations, as shown in \autoref{lump:figure:KLDMonotonicity_Combined}, but  we can not guarantee monotonicity in general.
\end{myRemark}

\setcounter{equation}{0}

\section{Discussions}
\label{lump:sec:discussion}

The idea of using Markov chain lumpability for model reduction has been discussed in the literature for some years now. For instance, the authors in  \cite{Simon2011Exact,Simon2012MeanField,kiss2017mathematics} considered epidemiological scenarios, focussing mainly on binary dynamics.  More general Markovian agent-based models were considered in \cite{Banisch2016AMB}. Lumpability abstractions were applied to rule-based systems in \cite{FERET2012Lumpability} from a theoretical computer science perspective. While model reduction is one of the main purposes of  lumpability, it is not the only one. In a recent paper \cite{katehakis_smit_2012}, the authors identify a class of Markov chains, which they call successively lumpable and for which the stationary probabilities can be computed successively by computing stationary probabilities of a cleverly constructed sequence of Markov chains (typically on much smaller state spaces).

\paragraph{Coverings and colour refinements}
For undirected graphs, a  notion similar to our notion of local symmetry is called a \emph{covering}  \cite{Angluin80}. However, in general, finding coverings is computationally challenging \cite{kratochv1998complexity}. In our formulation,  undirected graphs are to be treated as directed graphs with an edge set $E$ satisfying $(i,j)\in E \iff (j,i)\in E$. The second notion that is similar to our approach is that of colour refinement \cite{BerkholzBonsmaGrohe13,ArvindKoblerRattanVerbitsky16}.  In order to draw analogy, we think of the local states as colours, \ie, we have a $K$-colouring of $G$, and require isomorphisms to be colour-preserving. The objective  is to devise a colouring method (given the initial colouring) that is \emph{stable} in that two vertices with the same colour have identically coloured neighbourhoods. Note that a colouring naturally induces an equivalence relation on $V$. With successive refinement of colouring, we can construct equitable partitions of $V$ in much the same way we did with local symmetry. The equitable partitions, in turn, can be used to yield approximately lumpable partitions of $\mathcal{X}^N$. Colour refinements are convenient and   are often used as a simple isomorphism check. However, a limitation of this approach is that colour refinements are insufficient to find local isomorphisms in certain graphs such as regular graphs. In general, a graph $G$ is said to be amenable to colour refinement if it is distinguishable from any other graph $H$ via colour refinement. A number of classes of graphs are known to be amenable \cite{ArvindKoblerRattanVerbitsky16}, \eg, unigraphs, trees. It is also known \cite{BabaiErdosSelkow80} that  Erd\"os-R\'enyi  random graphs are amenable with high probability.

\paragraph{Regular graphs}
Large regular graphs, in general,  can exhibit  different dynamics on them. Since the vertices have similar neighbourhoods, our local symmetry will not be able to distinguish between them. This may lead to  poor lumping. Increasing the order of local symmetry will avoid such issues.  A theoretical analysis of this special case of regular graphs  is planned for future work.

\paragraph{Computation of the stationary distribution}
Note that computation of the stationary distribution itself is  cumbersome for Markov chains on large state spaces. In many cases, the transition matrix is sparse, which makes available a host of numerical tools developed for sparse matrices. There are also numerical algorithms \cite{stewart2000numerical}, such as the  Courtois' method \cite{courtois2014decomposability} or the Takahasi's iterative aggregation-disaggregation method \cite{takahashi1975lumping}, for computing the stationary distribution. In general, the efficiency of the Takahashi's algorithm depends on a good initial clustering of states. In our case, the computation is facilitated by the fact that the initial  quotient space $\mathcal{X}^N/\overset{1}\sim$ is expectedly a better partition than a random one.  In a recent paper \cite{kuntz2017rigorous}, the authors derive bounds on the stationary distribution (and moments) based on mathematical programming. In particular, when the stationary distribution is unique, they provide computable error bounds. Sampling-based techniques can also be used for this purpose. For instance, in \cite{hemberg2008dominated}, the authors provide an algorithm that combine Gillespie's algorithm with the  Dominated Coupling From The Past (DCFTP) techniques to provide guaranteed sampling from the stationary distribution.

\paragraph{Markov chain enlargement}
An interesting concept closely related to aggregation  is Markov chain enlargement. There are many examples where enlargement of the state space of a Markov chain can be computationally beneficial in that it can significantly reduce the mixing time. See \cite{Chen:1999:LMC,apers2017lifting} for a discussion on how splitting up states of a Markov chain can speed up mixing. This has implications for the performance of statistical inference algorithms that rely on the mixing of some Markov chain, and also for optimisation algorithms such as the alternating direction method of multipliers (ADMM). In  \cite{francca2017markov}, the authors show that, for certain objective functions, the distributed ADMM algorithm can indeed be seen as a lifting of the gradient descent algorithm.

\paragraph{CTBNs and SANs}
The Markovian agent-based model that we consider in this paper belongs to a more general class of models known as interacting particle systems (IPS) in the probability literature. The tools developed in this paper are expected to find applications beyond what has been discussed here and are immediately applicable to many of the traditional IPS models arising from statistical physics, population biology and social sciences. Such models include contact processes, voter models, exclusion models.  The MABM model discussed in the present paper is also closely related to continuous time Bayesian networks (CTBNs) \cite{nodelman2002CTBN} and stochastic automata networks (SANs) \cite{buchholz2004kronecker}. To be specific, the local intensity functions defined in \autoref{lump:eq:localIntensityFunction} constitute the Conditional Intensity Matrix (CIM) in  \cite{nodelman2002CTBN}. These CIMs can be then combined into $Q$ by the ``amalgamation'' operation. Another approach that is popular in SAN literature is to give $Q$ a Kronecker representation \cite{buchholz2004kronecker}. We expect the present endeavour will benefit and bridge the gap between the different communities that make use of the agent-based models.

%
%



\setcounter{equation}{0}
\appendix

\section{}
\label{lump:sec:appendixA}

\begin{proof}[Proof of Proposition~\ref{lump:proposition:permutation}]
  It can be verified that $ \{\mathcal{\tilde{Y}}_1, \mathcal{\tilde{Y}}_2,\ldots, \mathcal{\tilde{Y}}_M  \} $ indeed forms a partition of $\mathcal{Y}$. Let us denote the transition rate matrix of $Z$ by $\tilde{A}=((\tilde{a}_{i,j}))$, where $  \tilde{a}_{i,j} = a_{f^{-1}(i), f^{-1}(j)   } $, and $f^{-1}$ is the inverse of $f$ in $\SymmetryGroup{ \mathcal{Y} } $. The proof will be complete if we show that the linear system $\dot{z}=z\tilde{A}$ is lumpable. 
  Pick $\mathcal{\tilde{Y}}_i$, and $ \mathcal{\tilde{Y}}_j$ for $i \neq j$, and let $u, v \in \mathcal{\tilde{Y}}_i$ be arbitrarily chosen.  See that $u \in \mathcal{\tilde{Y}}_i$  implies $f^{-1}(u) \in \mathcal{Y}_i $.  Then,
  \begin{align*}
  \sum_{ l \in \mathcal{\tilde{Y}}_j  } \tilde{a}_{u, l}    =   \sum_{ l \in  \mathcal{\tilde{Y}}_j } a_{ f^{-1}(u),f^{-1}(l)} =   \sum_{ l \in  \mathcal{X}_j } a_{ s,l}  =   \sum_{ l \in  \mathcal{X}_j } a_{ t,l}  = \sum_{ l \in  \mathcal{\tilde{Y}}_j } a_{ f^{-1}(v),f^{-1}(l)}  =   \sum_{ l \in \mathcal{\tilde{Y}}_j  } \tilde{a}_{v, l} \, \eqcomma
  \end{align*}
where  $s=f^{-1}(u), t=f^{-1}(v) \in \mathcal{X}_i$ and the equality for $s$ and $t$ 
holds by virtue of the lumpability of $Y$. This verifies the Dynkin's criterion for $ \dot{z}=z\tilde{A} $. 
\end{proof}

\begin{proof}[Proof of Proposition~\ref{lump:proposition:fibration_1}]
 Let us first assume $x \in \fibre{y}$. In order to prove the vertices $x,y$ are locally symmetric, we construct an isomorphism  $g : N_1(x) \longrightarrow N_1(y)$ between $G[N_1(x)]$ and $G[N_1(y)]$ as follows:
 \begin{align}
   g(a) \defeq \mathsf{s}_G f_e^{-1} (a,x), \quad \forall a \in N_1(x) \eqstop
 \end{align}
Indeed, $f_e^{-1} (a,x)$ is an edge in $G[N_1(y)]$, and therefore, $g(a) \in N_1(y)$. In order to check whether $g$ is indeed an isomorphism, pick two vertices $a,b \in N_1(x)$ such that $(a,b) \in E$. If $b=x$, the assertion follows straightforwardly. Therefore, we consider $b\neq x$. Then, $(a,b) \in E$ implies the vertices $a,b,$ and $x$ form a triangle (see \autoref{lump:fig:fibration_localSymmetry}).

Since $f$ is a fibration,  $(f_v, f_e^{-1})$ is also a morphism because $f_v$ and $f_e^{-1}$ also commute with the source and target maps of $G$, \ie, $\mathsf{s}_G f_e^{-1} = f_v \mathsf{s}_G  $ and $ \mathsf{t}_G f_e^{-1} = f_v \mathsf{t}_G $. Now, let us consider the edge $(a,b)$ in $G[N_1(x)]$. Since $f$ is a fibration, there exists a unique edge $f_e^{-1} (a,b) = (c,d) \in E$ such that $f_e(c,d) =(a,b)$, where $d \in \fibre{b}$. Then,
\begin{align*}
  (c,d) = ( \mathsf{s}_G f_e^{-1}(a,b),  \mathsf{t}_G f_e^{-1} (a,b)   ) = (f_v(a),f_v(b)) =  ( \mathsf{s}_G f_e^{-1}(a,x),  \mathsf{t}_G f_e^{-1} (b,x)   )  \eqstop
\end{align*}
Therefore,  $g$ is indeed an isomorphism between $G[N_1(x)]$ and $G[N_1(y)]$ proving $x \overset{1}\sim y $.

Now, we prove the second part of the proposition. Let us assume $  x \overset{1}\sim y  $. In order to define a fibration $f=(f_v,f_e)$, let us first pick representatives for the equivalence classes of $ \overset{1}\sim  $. Let the injective map $\mathsf{r} :V \longrightarrow V$ define the representatives, that is, for each $x\in V$, we have $\coset{x}{1} = \coset{ \mathsf{r}(x)  }{1}  $. Then, consider the following maps
\begin{align*}
  f_v(x) \defeq \mathsf{r}(x), \forall x \in V, \text{ and } f_e(a,b)= ( g (a), g(b)  ) \eqcomma
\end{align*}
where $   g \in \Psi_1  $ is such that $ g (b) =  \mathsf{r}(b)$. Please note that the choice of $g$ depends on $(a,b)$.  The epimorphism $f$ defined above is indeed a fibration \cite{boldi2002fibrations}.
\end{proof}

\subsection*{Monotonicity fails for $P$-lifting}

It is intuitive that the monotonic decrease of KL divergence for finer partitions should carry over to lifting by the transition matrix. However, this is not the case as the following counterexample shows. Consider a transition probability matrix:
\begin{align*}
T=\begin{pmatrix}
0.10 & 0.10 & 0.07 & 0.16 & 0.13 & 0.20 & 0.04 & 0.20 \\
0.11 & 0.17 & 0.10 & 0.15 & 0.12 & 0.13 & 0.14 & 0.08 \\
0.07 & 0.13 & 0.10 & 0.14 & 0.09 & 0.02 & 0.41 & 0.04 \\
0.16 & 0.08 & 0.02 & 0.17 & 0.05 & 0.23 & 0.06 & 0.23 \\
0.07 & 0.12 & 0.20 & 0.17 & 0.22 & 0.21 & 0.01 & 0.00 \\
0.07 & 0.15 & 0.25 & 0.10 & 0.18 & 0.03 & 0.21 & 0.01 \\
0.14 & 0.07 & 0.20 & 0.14 & 0.10 & 0.10 & 0.07 & 0.18 \\
0.10 & 0.19 & 0.07 & 0.22 & 0.11 & 0.03 & 0.14 & 0.14
\end{pmatrix} \eqstop
\end{align*}
Now, consider  two partitions:$\{ \{1,2,3,4\},\{5,6,7,8\} \}$ and $\{\{1,2\},\{3,4\},\{5,6\},\{7,8\}\}$. Clearly the latter partition is a refinement of the first. However, when we use $P$-lifting, the first partition yields a KL divergence of $0.0019067$, while the second partition yields a higher KL divergence of $0.0308801$.

%

%


\section*{ACKNOWLEDGEMENTS}
This work  has been co-funded by the German Research Foundation~(DFG) as part of project C3 within the Collaborative Research Center~(CRC) 1053 -- MAKI.

\bibliographystyle{abbrv}
\bibliography{main}

\end{document}